\documentclass[a4paper,11pt]{article}
\usepackage[T1]{fontenc} 
\usepackage{amsmath,amsthm}
\usepackage[francais,english]{babel}
\usepackage[dvips]{graphicx}
\usepackage{todonotes}
\usepackage{color}
\usepackage{tikz}
\usetikzlibrary{arrows,positioning,shapes,matrix,backgrounds}
\usetikzlibrary{fit,calc,decorations.pathreplacing}

\usepgflibrary{arrows}
\usepackage{soul}
\usepackage{amsfonts,amssymb}
\usepackage{geometry}
\usepackage{hyperref}
\usepackage{stmaryrd}
\usepackage{fancyhdr}
\usepackage{url}
\usepackage{todonotes}
\usepackage{dsfont}
\usepackage[utf8]{inputenc} 
\usepackage{enumerate}

\theoremstyle{definition}

\newtheorem*{class}{Main Theorem 1}
\newtheorem*{maintheorem}{Main Theorem 2}

\newtheorem*{ack}{Acknowledgements}
\newtheorem*{thm*}{Theorem}

\newtheorem{prop}{Proposition}[section]

\newtheorem{lemma}[prop]{Lemma}

\newtheorem{thm}[prop]{Theorem}
\newtheorem{definition}[prop]{Definition}
\newtheorem{notation}[prop]{Notation}
\newtheorem{corollary}[prop]{Corollary}

\newtheoremstyle{pourlesremarques}{\topsep}{\topsep}{\normalfont}{}{\bfseries}{.}{ }{}
\theoremstyle{pourlesremarques}
\newtheorem{rem}[prop]{Remark}
\newtheorem*{rem*}{Remark}
\newtheoremstyle{pourlesexemples}{\topsep}{\topsep}{\normalfont}{}{\bfseries}{.}{ }{}
\theoremstyle{pourlesexemples}
\newtheorem{ex}[prop]{Example}

\def\presuper#1#2%
  {\mathop{}%
   \mathopen{\vphantom{#2}}^{#1}%
   \kern-\scriptspace%
   #2}

\newcommand{\St}{{\operatorname{St}}}
\newcommand{\GJ}{{\operatorname{GJ}}}
\newcommand{\RS}{{\operatorname{RS}}}
\newcommand{\ab}{{\operatorname{ab}}}

\newcommand{\gen}{{\operatorname{gen}}}

\def\Rep{\operatorname{Rep}}
\def\Irr{\operatorname{Irr}}
\newcommand{\Indec}{\operatorname{Indec}}
\newcommand{\Nilp}{\operatorname{Nilp}}
\newcommand{\J}{\mathrm{J}}

\newcommand{\Ind}{\operatorname{Ind}}
\newcommand{\ind}{\operatorname{ind}}

\newcommand{\LLC}{\operatorname{LLC}}

\newcommand{\Hom}{\operatorname{Hom}}
\renewcommand{\subset}{\subseteq}

\newcommand{\WD}{\mathrm{D}}
\newcommand{\W}{\mathrm{W}}
\newcommand{\G}{\mathrm{G}}
\newcommand{\GL}{\operatorname{GL}}
\newcommand{\tr}{\operatorname{tr}}
\newcommand{\Sp}{\operatorname{Spec}}

\newcommand{\Ker}{\mathrm{Ker}}

\newcommand{\w}{\varpi}

\renewcommand{\l}{\lambda}
\newcommand{\La}{\mathrm{L}}
\newcommand{\V}{\mathrm{V}}
\newcommand{\C}{\mathrm{C}}
\newcommand{\CV}{\mathrm{CV}}

\newcommand{\Ze}{\mathrm{Z}}
\newcommand{\F}{\mathbb{F}}
\newcommand{\Q}{\mathbb{Q}}
\newcommand{\N}{\mathbb{N}}
\newcommand{\m}{\mathfrak{m}}
\newcommand{\M}{\mathrm{Mat}}

\newcommand{\Z}{\mathbb{Z}}
\newcommand{\e}{\epsilon}
\newcommand{\1}{\mathbf{1}}

\def\End{\operatorname{End}}
\def\Iso{\operatorname{Iso}}
\def\acyc{\operatorname{acyc}}
\def\cyc{\operatorname{cyc}}

\def\ss{\operatorname{ss}}
\def\sc{\operatorname{sc}}
\def\c{\operatorname{c}}
\def\b{\operatorname{b}}
\def\tnb{\operatorname{tnb}}
\def\GL{\operatorname{GL}}
\def\Id{\operatorname{Id}}

\def\\Hom{\operatorname{\Hom}}
\def\Irr{\operatorname{Irr}}
\def\Fr{\operatorname{Fr}}

\def\dim{\operatorname{dim}}
\def\St{\operatorname{St}}

\def\Ql{\overline{\mathbb{Q}_{\ell}}}

\def\Rep{\operatorname{Rep}}
\def\Zl{\overline{\mathbb{Z}_{\ell}}}
\def\Fl{\overline{\mathbb{F}_{\ell}}}
\def\Fr{\mathrm{Frob}}
\def\Gal{\operatorname{Gal}}

\def\Vect{\operatorname{Vect}}

\def\leq{\leqslant}
\def\geq{\geqslant}

\def\Rep{\operatorname{Rep}}
\newcommand{\D}{\Delta}

\def\presuper#1#2%
  {\mathop{}%
   \mathopen{\vphantom{#2}}^{#1}%
   \kern-\scriptspace%
   #2}
\makeatletter
\DeclareRobustCommand{\rvdots}{%
  \vbox{
    \baselineskip4\p@\lineskiplimit\z@
    \kern-\p@
    \hbox{.}\hbox{.}\hbox{.}  \hbox{.}
  }}
\makeatother

\title{The~$\ell$-modular local Langlands correspondence and local factors}
\author{R. Kurinczuk\footnote{Robert Kurinczuk, Department of Mathematics, Imperial College London, SW7 2AZ. U.K. \newline Email:~robkurinczuk@gmail.com, Tel: +44(0)7921 221967}, N. Matringe\footnote{Nadir Matringe, Universit\'e de Poitiers, Laboratoire de Math\'ematiques et Applications,
T\'el\'eport 2 - BP 30179, Boulevard Marie et Pierre Curie, 86962, Futuroscope Chasseneuil Cedex. France. \newline Email:~Nadir.Matringe@math.univ-poitiers.fr}}
\date{\today}
\begin{document}
\maketitle

\begin{abstract}
\noindent   Let~$F$ be a non-archimedean local field of residual characteristic~$p$, $\ell\neq p$ be a prime number, and~$\W_F$ the Weil group of~$F$.   We classify the indecomposable~$\W_F$-semisimple Deligne~$\Fl$-representations in terms of the irreducible~$\Fl$-representations of~$\W_F$, and extend constructions of~Artin--Deligne~local factors to this setting.  Finally, we define a variant of the~$\ell$-modular local Langlands correspondence which satisfies a preservation of local factors statement for generic representations.
\end{abstract}
\section{Introduction}
Let~$F$ denote a non-archimedean local field of residual cardinality~$q$ and residual characteristic~$p$.  Let~$\ell$ be a prime different to~$p$. We consider only smooth representations of locally profinite groups, and call them 
$\ell$-adic when they act on $\Ql$-vector spaces, and $\ell$-modular when they act on $\Fl$-vector spaces.  Let $\W_F$ denote the Weil group of $F$. 

The local Langlands correspondence $\LLC$ for~$\GL_n(F)$ is a canonical bijection between the set of isomorphism classes of $\ell$-adic irreducible representations of~$\GL_n(F)$ and the set of isomorphism classes of~$\ell$-adic~$n$-dimensional~$\W_F$-semisimple Deligne representations, generalizing the Artin reciprocity map of local class field theory.  A nice property of LLC is that the Rankin--Selberg local factors of a pair of irreducible~$\Ql$-representations of~$\GL_n(F)$ and~$\GL_m(F)$, and the Artin--Deligne local factors of the corresponding tensor product of representations of~$\W_F$ are equal, and moreover this condition characterizes $\LLC$ completely.

In \cite{Vbook}, Vign\'eras develops the theory of modular representations of reductive~$p$-adic groups over~$\Fl$.  For general linear groups, this culminates in the~$\ell$-modular local Langlands correspondence \cite{Viginv}: a bijection between the set of isomorphism classes of~$\ell$-modular irreducible representations of~$\GL_n(F)$ and the set of isomorphism classes of $\ell$-modular~$n$-dimensional~$\W_F$-semisimple Deligne representations with nilpotent Deligne operator. Vign\'eras characterizes her correspondence by compatibility with LLC and congruences, although not naively (see Section \ref{section Vcorresp} where we recall Vign\'eras' results precisely).  

The theory of Rankin--Selberg local factors of Jacquet, Shalika and Piatetski-Shapiro has a natural extension at least to $\ell$-modular generic representations of~$\GL_n(F)$ and~$\GL_m(F)$, \cite{KM17}.  However, via the~$\ell$-modular local Langlands correspondence these factors do not agree with the factors of Artin--Deligne.  

In this work, we classify the~$\ell$-modular indecomposable~$\W_F$-semisimple Deligne representations, extend the definitions of Artin--Deligne factors to this setting, and define an~$\ell$-modular local Langlands correspondence where in the generic case, the Rankin--Selberg factors of representations on one side equal the Artin--Deligne factors of the corresponding representations on the other.   

We now recall our definitions and conventions, and state our results precisely.  Let~$R$ be an algebraically closed field of characteristic~$\ell$ different from~$p$. Let~$\nu:\W_F\rightarrow R^\times$ be the unique character trivial on the inertia subgroup of~$\W_F$ and sending a geometric Frobenius element to~$q^{-1}$.  Fix a nontrivial character~$\psi:F\rightarrow R^\times$.

A \emph{Deligne~$R$-representation} of~$\W_F$ is a pair~$(\Phi,U)$ with~$\Phi$ a finite dimensional smooth representation of~$\W_F$ and~$U\in\Hom_{\W_F}(\nu\Phi,\Phi)$ the associated Deligne operator.  We say a Deligne~$R$-representation~$(\Phi,U)$ of~$\W_F$ is \emph{nilpotent} if~$U$ is a nilpotent endomorphism of the vector space of~$\Phi$, and~\emph{$\W_F$-semisimple} if~$\Phi$ is a semisimple representation of~$\W_F$.  We write~$(\Phi,U)^\vee$ for the \emph{dual} of~$(\Phi,U)$, see Definition \ref{dualdef}.  All Deligne~$\Ql$-representations of~$\W_F$ are nilpotent, however as we shall soon see this is not the case for Deligne~$\Fl$-representations.  

We say that indecomposable Deligne~$R$-representations~$(\Phi,U),(\Phi',U')$ of~$\W_F$ are \emph{equivalent} if 
there exists~$\lambda\in R^\times$ such that~$(\Phi,\lambda U)\simeq(\Phi',U')$, and we extend this definition to general Deligne~$R$-representations of~$\W_F$ thanks to the Krull--Schmidt theorem (see Definition \ref{equivalence} and Remark \ref{remark Krull}).  

Before stating our classification results, we need to define two fundamental examples. For $\Phi$ an isomorphism class of an $R$-representation $\W_F$, we denote by $\Z_\Phi$ the set $\{\nu^k\Phi,\ k\in \Z\}$. 
\begin{enumerate}[{(1)}]
\item Let~$\Psi$ be an irreducible~$\Fl$-representation of~$\W_F$, then there is a minimal positive integer~$o(\Psi)$ such that we have an isomorphism~$\nu^{o(\Psi)}\Psi$ to~$\Psi$.  Let~$I$ be such an isomorphism.  Define a Deligne~$\Fl$-representation~$C(\Psi,I)=(\Phi(\Psi),C_I)$ of~$\W_F$  by
\begin{align*}
\Phi(\Psi)&=\bigoplus_{k=0}^{o(\Psi)-1} \nu^{k}\Psi\\
C_I(x_0,\dots,x_{o(\Psi)-1})&=( I(x_{o(\Psi)-1}),x_0,\dots,x_{o(\Psi)-2}).
\end{align*}
In Lemma \ref{irreducible}, we show that~$C(\Psi,I)$ is irreducible.  The equivalence class of~$C(\Psi,I)$ depends only on~$\Z_{\Psi}$, and we denote it by~$C(\Z_\Psi)$.  
\item For a positive integer~$r$, we define a~$\W_F$-semisimple nilpotent Deligne~$\Fl$-representation
 \[[0,r-1]=(\Phi(r),N(r))\] by
\begin{align*}
\Phi(r)&=\bigoplus_{k=0}^{r-1} \nu^k,\\
N(r)(x_0,\ldots,x_{r-1})&=(0,x_0,\ldots,x_{r-2}).\end{align*}
In Lemma \ref{segment indec}, we show that~$[0,r-1]$ is indecomposable.
\end{enumerate}

We prove the following classification:

\begin{class}[Theorems \ref{simple} and \ref{indec classif}]
Let~$(\Phi,U)$ be an indecomposable $\W_F$-semisimple Deligne~$\Fl$-representations.
\begin{enumerate}[{(1)}]
\item There exists an irreducible~$\Fl$-representation~$\Psi$ of~$\W_F$ and a positive integer~$r$ such that either:
\begin{enumerate}[{(a)}]
\item\label{casea}\[\displaystyle (\Phi,U)\simeq [0,r-1]\otimes \Psi;\]
\item\label{caseb} there exists an isomorphism~$I$ from~$\nu^{o(\Psi)}\Psi$ to~$\Psi$ such that
\[\displaystyle (\Phi,U)\simeq [0,r-1] \otimes C(\Psi,I).\]
\end{enumerate}
\item Under the notation of the last part,~$\pi$ is irreducible if and only if~$r=1$.  
\item In case (\ref{casea}) the isomorphism class of~$(\Phi,U)$ coincides with its equivalence class, and determines the isomorphism class of~$\Psi$ and~$r$ are uniquely. 
\item In case (\ref{caseb}) the equivalence class $[\Phi,U]$ of~$[0,r-1]\otimes C(\Psi,I)$ only depends on $r$ and $\Z_\Psi$ and we denote it by 
$[0,r-1]\otimes C(\Z_\Psi)$. The datum $(r,\Z_\Psi)$ is uniquely determined by $[\Phi,U]$.
\end{enumerate}
\end{class}

Let~$(\Phi,U)$ be a~$\W_F$-semisimple Deligne~$R$-representation of~$\W_F$.  For an indeterminant~$X$, we put
\[L(X,(\Phi,U))=\det((\Id-X\Phi(\Fr))\mid_{\Ker(U)^{I_F}})^{-1},\]
it is an Euler factor.  Using results of~\cite{Deligne73}, we associate to~$(\Phi,U)$ a local~$\gamma$-factor~$\gamma(X,(\Phi,U),\psi)$, which does not see the operator~$U$, and put
\[\e(X,(\Phi,U),\psi)=\gamma(X,(\Phi,U),\psi)\frac{L(X,(\Phi,U))}{L(q^{-1}X^{-1},(\Phi,U)^{\vee})}.\]
Of course, these definitions coincide with the standard ones when~$R=\Ql$, as well as in the case $R=\Fl$ when 
$(\Phi,U)$ is of the form $(\Phi,0)$, see Section \ref{Sectionlocalfactors}.  

Following the~$\Ql$-case, we wish to define Artin--Deligne local factors of pairs of~$\W_F$-semisimple Deligne~$\Fl$-representations of~$\W_F$ via their tensor product.  However, there are two immediate obstacles:
\begin{enumerate}[{(1)}]
\item The tensor product of semisimple representations of~$\W_F$ is not necessarily semisimple.  We give an explicit example in Example \ref{counter examples} (\ref{Part3}). 
\item The tensor product does not preserve equivalence, see Example \ref{non equiv tensor}.
\end{enumerate}
Both of these problems have natural solutions and we define the semisimple tensor product~$\otimes_{\ss}$ of~$\W_F$-semisimple Deligne~$\Fl$-representations of~$\W_F$ in Section \ref{tensor section}.  
%
%

Finally, we move on to our results on the~$\ell$-modular local Langlands correspondence.  We call an isomorphism class of~$\W_F$-semisimple nilpotent Deligne~$\Fl$-representations of~$\W_F$ a \emph{$\V$-parameter}, and denote by~$\V$ the bijection of~\cite{Viginv}:
\begin{center}
\begin{tikzpicture}[node distance=2cm, auto,  decoration={brace}]
	\node[
           text width=14em,
           minimum height=2em,
           text centered] (GL) {Irreducible $\Fl$-representations\\ of~$\GL_n(F)$ up to isomorphism};
 	\node[right=of GL,
           text width=13.5em,
           minimum height=2em,
           text centered] (Gal) {$\V$-parameters of dimension~$n$.}
	   edge[>=stealth,<-,
           shorten <=10pt,
           shorten >=10pt,] node[above]{$\V$} (GL) ;
            \draw [decorate,line width=1pt]  ([yshift=0pt]GL.south west) -- ([yshift=0pt]GL.north west);
    \draw [decorate,line width=1pt]  ([yshift=0pt]GL.north east) -- ([yshift=0pt]GL.south east);
         \draw [decorate,line width=1pt]  ([yshift=0pt]Gal.south west) -- ([yshift=0pt]Gal.north west);
     \draw [decorate,line width=1pt]  ([yshift=0pt]Gal.north east) -- ([yshift=0pt]Gal.south east);
\end{tikzpicture}
\end{center}

For a pair~$\pi,\pi'$ of generic~$\Fl$-representations of~$\GL_n(F),\GL_m(F)$ respectively we denote by
\[L(X,\pi\times\pi'),\quad \e(X,\pi\times\pi',\psi),\quad \gamma(X,\pi\times\pi',\psi),\]
the local factors defined in~\cite{KM17}.

The motivation for our next results is that the correspondence~$\V$ does not preserve local factors of generic representations, for example it is not true that $L(X,\pi\times \pi')=L(X,\V(\pi)\otimes_{\ss}\V(\pi'))$, see Example \ref{non preservation of local factors}.  In Definition \ref{CVdefinition}, we define an injective map
%
\begin{center}
\begin{tikzpicture}[node distance=0cm, auto, decoration={brace}]
	\node[
           text width=3em,
           minimum height=1em,text centered] (CV) {$\C\V:$};
	\node[right=of CV,
           text width=6em,
           minimum height=1em,
           text centered] (G) {$\V$-parameters};
           \node[right=of G,
           text width=1.5cm,
           minimum height=1em,
           text centered] (GL) {};           
 	\node[right=of GL,
           text width=17em,
           minimum height=1em,
           text centered] (Gal) {Equivalence classes of $\W_F$-semisimple Deligne~$\Fl$-representations of~$\W_F$}
	   edge[>=stealth,<-,
           shorten <=10pt,
           shorten >=10pt,] node[above]{} (G) ;
 \draw [decorate,line width=1pt]  ([yshift=0pt]G.south west) -- ([yshift=0pt]G.north west);
    \draw [decorate,line width=1pt]  ([yshift=0pt]G.north east) -- ([yshift=0pt]G.south east);
         \draw [decorate,line width=1pt]  ([yshift=0pt]Gal.south west) -- ([yshift=0pt]Gal.north west);
     \draw [decorate,line width=1pt]  ([yshift=0pt]Gal.north east) -- ([yshift=0pt]Gal.south east);
\end{tikzpicture}
\end{center}

which is not the natural inclusion, we call an element in its image a~\emph{$\C$-parameter.}  We can now state our main result, the first three properties of which are immediate consequence of the analogues for $\V$ and the definition of $\CV$:
%

\begin{maintheorem}\label{C}[Theorem \ref{preservation of local factors}]
For positive integers~$n$, the bijections~$\C=\C\V\circ \V$:
\begin{center}
\begin{tikzpicture}[node distance=2cm, auto, decoration={brace}]
	\node[
           text width=14em,
           minimum height=1em,
           text centered] (GL) {Irreducible $\Fl$-representations\\ of~$\GL_n(F)$ up to isomorphism};
 	\node[right=of GL,
           text width=13em,
           minimum height=1em,
           text centered] (Gal) {$\C$-parameters of dimension~$n$}
	   edge[>=stealth,<-,
           shorten <=10pt,
           shorten >=10pt,] node[above]{$\C$} (GL) ;
 \draw [decorate,line width=1pt]  ([yshift=0pt]GL.south west) -- ([yshift=0pt]GL.north west);
    \draw [decorate,line width=1pt]  ([yshift=0pt]GL.north east) -- ([yshift=0pt]GL.south east);
         \draw [decorate,line width=1pt]  ([yshift=0pt]Gal.south west) -- ([yshift=0pt]Gal.north west);
     \draw [decorate,line width=1pt]  ([yshift=0pt]Gal.north east) -- ([yshift=0pt]Gal.south east);
\end{tikzpicture}
\end{center}
satisfy the following properties:

Let~$\pi$ be an irreducible~$\Fl$-representation of~$\GL_n(F)$.
\begin{enumerate}[{(1)}]
\item For all characters~$\chi:\GL_n(F)\rightarrow \Fl^\times$, 
\[\C(\chi\pi)=\chi\C(\pi).\]
\item Letting~$c_{\pi}$ denote the central character of~$\pi$, then using local class field theory
\[c_{\pi}=\det(\C(\pi)).\]
\item Commutation with (smooth) duals:
\[C(\pi^\vee)=C(\pi)^\vee.\]
\item If~$\pi$ is generic, for all generic~$\Fl$-representations~$\pi'$ of~$\GL_m(F)$ with~$1\leq m$,
\begin{align*}
L_{}(X,\pi\times \pi')&=L(X,\C(\pi)\otimes_{\ss} \C(\pi')),\\
\gamma_{}(X,\pi\times \pi',\psi)&=\gamma(X,\C(\pi)\otimes_{\ss} \C(\pi'),\psi),\\
\e_{}(X,\pi\times \pi',\psi)&=\e(X,\C(\pi)\otimes_{\ss} \C(\pi'),\psi).\end{align*}
\end{enumerate}
\end{maintheorem}

In Example \ref{Cexample}, we give examples of the~$\C$-correspondence.

As a corollary of our results, if one defines the~$L,\epsilon,\gamma$-factors of a~$\V$-parameter~$\Phi$ to be the local factors of the corresponding~$\C$-parameter~$\C\V(\Phi)$, then we have a \emph{preservation of local factors} for generic representations result for~$\V$. However, as this construction goes via the associated~$\C$-parameter it feels more natural to us to state the result in terms of~$\C$.

It is tempting to say that~$\C$ should preserve local factors of pairs of irreducible~$\Fl$-representations of~$\GL_n(F)$ beyond the generic setting, for example the \emph{Godement--Jacquet local factors} defined in~\cite{MinguezZeta}.  This holds for~$n=2$ for the factors of [ibid.], thanks to the explicit computations of~M\'inguez in this case.

Finally, motivated by the corresponding result for~$\LLC$, one might wonder to what extent the list of properties of Theorem 2 characterize the correspondence.  We leave this question for future work.

\begin{ack}
We thank David Helm and Alberto M\'inguez for useful conversations.  This work was started at the Imperial College London and continued at the Universit\'e de Poitiers, and the authors would like to thank them for their hospitality.  We would like to thank the Heilbronn Institute for Mathematical Research for supporting these visits.  The second author benefited from support from the grant ANR-13-BS01-0012 FERPLAY.
\end{ack}

\section{General Notations}

Let~$F$ be a non-archimedean local field with finite residue field~$k_F$ of characteristic $p$ and cardinality~$q$. For a positive integer~$n$, we denote by~$F_n$ the unique (up to isomorphism) unramified extension of~$F$ of degree~$n$. 

Let~$\ell$ be a prime number different from~$p$.  We fix an algebraic closure~$\Ql$ of the~$\ell$-adic numbers~$\Q_{\ell}$.  We denote by~$\Zl$ the ring of integers of~$\Ql$ and by~$\m$ its maximal ideal. We put~$\Fl=\Zl/\m$, it is an algebraic closure of the finite field~$\F_{\ell}$ with~$\ell$ elements.  

We consider smooth representations of locally profinite groups on~$\Ql$ and~$\Fl$-vector spaces, and the connections between them.  Henceforth, all representations considered are implicitly assumed to be smooth and, for convenience of stating results which apply to both cases~$\Ql$ and~$\Fl$, we let~$R$ denote either field.  We abbreviate ``representation on an~$R$-vector space'' to~\emph{$R$-representation}.  When~$R=\Fl$ we say we are in the \emph{$\ell$-modular case}, and when~$R=\Ql$ we say we are in the~\emph{$\ell$-adic case}.

Let~$H$ be a locally profinite group.  Let~$\Rep(H,R)$ denote the abelian category of~$R$-representations of~$H$, and~$\Irr(H,R)$ denote the set of isomorphism classes of irreducible~$R$-representations of~$H$.  Let~$\pi,\pi'$ be~$R$-representations of~$H$.  We denote by~$\pi^\vee$ the (smooth) contragredient of~$\pi$, and put
\[\End_H(\pi)= \Hom_H(\pi,\pi),\text{ and }\Iso_H(\pi,\pi')=\{\phi\in\Hom_H(\pi,\pi'):\phi\text{ is bijective}\}.\]

Let~$K$ be a closed subgroup of~$H$. We fix a square root of~$q$ in~$\Zl$ and take its image in $\Fl$ so we have 
a fixed square root of $q$ in $R$, which is compatible with reduction modulo $\ell$. We denote by
\[\Ind_K^H:\Rep(K,R)\rightarrow \Rep(H,R)\] the usual normalized (with respect to this square root of $q$) induction functor.

Strictly speaking an~$R$-representation consists of a pair~$(\Phi,V)$ with~$V$ an~$R$-vector space and 
\[\Phi:H\rightarrow \GL_R(V)\] a group homomorphism.  However, as we have done already, we will often denote the pair~$(\Phi,V)$ just by~$\Phi$, and in this case we will use~$V_{\Phi}$, or even~$\Phi$, to denote the underlying vector space upon which~$H$ acts via~$\Phi$.   We also make no distinction between an~$R$-representation and its isomorphism class, and will write~$\Phi\in\Rep(H,R)$ for~$\Phi$ is an object in the category~$\Rep(H,R)$, i.e.~an~$R$-representation of~$H$.  

\section{Representations of $\W_F$}\label{Weil group}

\subsection{Notations}

We refer to \cite[Chapter 7]{BHbook} for the definitions and facts stated here concerning Weil groups. We fix a separable algebraic closure~$\overline{F}$ of~$F$, and will suppose that all finite extensions we consider are contained in~$\overline{F}$.  For a finite extension~$E/F$ we let~$\G_E=\Gal(\overline{F}/E)$ denote the \emph{absolute Galois group} of~$E$;~$I_E$ denote the \emph{inertia subgroup} of~$\G_E$;~$P_E$ denote the \emph{wild inertia subgroup}, it is the pro-$p$ Sylow subgroup of~$I_E$;~$\W_E$ denote the \emph{Weil group} of~$E$.  We fix a geometric Frobenius element~$\Fr$ in~$\W_F$.  We have 
\[\W_F=I_F\rtimes \Fr^\Z,\] and, in particular,~$\W_F$ is unimodular as~$I_F$ is compact.

If~$E$ is a finite extension of~$F$, and~$\pi$ is a representation of~$\W_F$, we denote by~$\pi_E$ the restriction of~$\pi$ to~$\W_E$.   An~$R$-representation of~$\W_F$ is called \textit{unramified} if it is trivial on~$I_F$. If such a representation is irreducible, then it is necessarily a character as~$\W_F/I_F$ is abelian. In fact, if~$\Psi$ is an irreducible representation of~$\W_F$ such that~$\Psi^{I_F}$ (the~$I_F$-fixed subspace of~$\Psi$) is nonzero, then~$\Psi^{I_F}$ is a nonzero~$\W_F$-invariant subspace of~$\Psi$, hence~$\Psi=\Psi^{I_F}$ is an unramified character.  We denote by~$\nu$ the unramified character of~$\W_F$ which satisfies~$\nu(\Fr)=q^{-1}$, and by~$X^u(\W_F)$ the group of unramified characters of~$\W_F$.  

An~$R$-irreducible representation of~$\W_F$ is called \textit{tamely ramified} if it is trivial on~$P_F$. We denote by~$\Irr^{\tr}(\W_F, R)$ the set of isomorphism classes of tamely ramified irreducible representations of~$\W_F$.

Let~$\Psi\in \Irr(\W_F,R)$, we put
\[\Z_\Psi=\{\nu^k\Psi,\ k\in \Z\}\] and call such a set an \textit{irreducible line}.  
When~$R=\Fl$, the set~$\Z_\Psi$ is finite, and we denote by~$o(\Psi)$ its cardinality 
\[o(\Psi)=|\Z_\Psi|.\]
The integer~$o(\Psi)$ clearly divides the order of~$\nu$ which is the order of~$q$ modulo~$\ell$.  In particular 
$o(\Psi)$ divides~$\ell-1$ and is hence prime to~$\ell$. We say that~$\Psi$ is \textit{banal} if~$o(\Psi)>1$.

\subsection{Irreducible representations of $\W_F$}

We now recall a description of the irreducible~$R$-representations of~$\W_F$ as induced representation which is well suited to studying congruences between~$\Ql$-representations. 

Let~$E$ be a finite Galois extension of~$F$ and~$\pi\in\Irr(\W_E, R)$. For~$\sigma\in \Gal(E/F)\simeq \W_F/\W_E$, 
we denote by~$\pi^\sigma$ the~$R$-representation of~$\W_E$ defined on the same underlying space as~$\pi$ by
\[\pi^\sigma(w)=\pi(\sigma w  \sigma^{-1}),\] 
for all~$w\in \W_E$. 

We say that~$\pi$ is~$\Gal(E/F)$-regular if the set \[\{\pi^\sigma: \sigma\in \Gal(E/F)\}\subset \Irr(\W_E, R)\] is 
of cardinality~$|\Gal(E/F)|$. 

\begin{lemma}\label{mackey irred criterion}
Let~$E$ be a finite Galois extension of~$F$ and~$\tau$ an~$R$-representation of~$\W_F$. Then~$\Ind_{\W_E}^{\W_F}(\tau)$ is irreducible if and only if~$\tau$ is irreducible and~$\Gal(E/F)$-regular.
\end{lemma}
\begin{proof}
If~$\Ind_{\W_E}^{\W_F}(\tau)$ is irreducible then~$\tau$ is as well, by exactness of the induction functor. Moreover its endomorphism algebra is isomorphic to~$R$ by Schur's lemma. However by Frobenius reciprocity and Mackey theory we have: 
\begin{align*}
\End_{W_F}(\Ind_{\W_E}^{\W_F}(\tau))&\simeq \Hom_{W_E}(\Ind_{\W_E}^{\W_F}(\tau)\vert_{W_E},\tau)\\
&\simeq \Hom_{W_E}\left({\textstyle\bigoplus_{\sigma\in \Gal(E/F)}} \tau^\sigma,\tau\right)\\
&\simeq \prod_{\sigma\in \Gal(E/F)}\Hom_{W_E}(\tau^\sigma,\tau).\end{align*} This latter ring is isomorphic to~$R$ 
if and only if~$\tau$ is regular. 

Conversely, if~$\tau$ is irreducible and~$\Gal(E/F)$-regular, let~$W$ be a nonzero~$\W_F$-stable subspace of 
$\Ind_{\W_E}^{\W_F}(\tau)$. Then it is a~$\W_E$-stable subspace of~$\bigoplus_{\sigma\in \Gal_F(E)} \tau^\sigma$ by Mackey theory.  In particular as the regularity assumption implies that this direct sum is the decomposition into isotypic components of~$\Ind_{\W_E}^{\W_F}(\tau)\vert_{\W_E}$, the space~$W$ contains~$\tau^\sigma$ for some~$\sigma$. As it is stable under~$\W_F$, it contains all~$\tau^\sigma$, hence it is equal to~$\Ind_{\W_E}^{\W_F}(\tau)$.
\end{proof}

We recall the following well-known description of the irreducible~$R$-representations of~$\W_F$. They are a consequence of \cite[Lemma 2.3, Proposition 2.1]{BHmodl} in the~$\ell$-adic case, and of \cite[2.6]{Vpartialcorrespondence} in the modular case.

\begin{thm}\label{irreps of the Weil group}
\begin{itemize}
\item\label{Part1Irreps} Take $\Psi\in \Irr^{\tr}(\W_F, R)$ of dimension $n$, then there is a $\Gal_{F}(F_n)$-regular tamely ramified character $\chi$ of $\W_{F_n}$ such that 
\[\Psi=\Ind_{\W_{F_n}}^{\W_F}(\chi),\] the character $\chi$ being unique up to conjugation by $\Gal_{F}(F_n)$.
\item\label{Part2Irreps} Take $\Psi\in \Irr(\W_F,R)$, then there exist a (finite) tamely ramified extension $E$ of $F$, 
$\Psi^{\tr}\in \Irr^{\tr}(\W_E, R)$, and $\tau\in \Irr(\W_E, R)$ which restricts irreducibly to $P_F$, such that 
\[\Psi=\Ind_{\W_E}^{\W_F}(\Psi^{\tr}\otimes \tau),\] the representation $\Psi^{\tr}\otimes \tau$ being unique up to 
to conjugation by $\Gal_{F}(E)$. Moreover for fixed $\Psi$ and $\tau$, the representation $\Psi^{\tr}$ is unique.
\end{itemize}
\end{thm}

For $\Psi\in \Irr(\W_F,R)$, we denote by $R(\Psi)$ the group of unramified characters of $\W_F$ fixing $\Psi$:
\[R(\Psi)=\{\mu\in X^u(\W_F),\ \chi\Psi=\Psi\}.\]

\begin{corollary}\label{ramification group}
Write $\Psi\in \Irr(\W_F,R)$ under the form $\Psi=\Ind_{\W_E}^{\W_F}(\Psi^{\tr}\otimes \tau)$ as in the second point of Theorem \ref{irreps of the Weil group}, and 
write $\Psi^{tr}=\Ind_{\W_{E_n}}^{\W_E}(\chi)$ for $\chi$ a $\Gal_{E}(E_n)$-regular tamely ramified character $\chi$ of $\W_{E_n}$. 
Set $r=[E_n:F]$, and $e$ the ramification index of $E_n/F$. The map $\mu\mapsto \mu(\Fr)$ is an isomorphism between $R(\pi)$ and the group of $r/e$-th roots of unity in 
$R^\times$.\end{corollary}
\begin{proof}
Let $e$ be the ramification index of $E/F$ and $t=[E:F]$. One can take $\Fr_{E}=\Fr^{t/e}$ (where $\Fr_E$ stands for a geometric Frobenius element in $\W_E$). Then an unramified character $\mu$ belongs to 
$R(\Psi)$ if and only if \[\Ind_{\W_E}^{\W_F}(\mu_E\Psi^{\tr}\otimes \tau)=\Ind_{\W_E}^{\W_F}(\Psi^{\tr}\otimes \tau),\] which is equivalent to 
\[\mu_E \Psi^{\tr}= \Psi^{\tr}.\]
This is in turn equivalent to $\mu_{E_n}\chi$ is conjugate to $\chi$ by $\Gal_{E}(E_n)$, i.e. $\mu_{E_n}$ is of the form 
$\chi^{\sigma}/\chi$ for $\sigma\in \Gal_{E}(E_n)$. However $\Gal_{E}(E_n)$ is generated by $\Fr_{E_n}$, but as one can 
take $\Fr_{E_n}=(\Fr_{E})^n=\Fr^{nt/e}$, a character of the form 
$\chi^{\sigma}/\chi$ is trivial on $\Fr_{E_n}$, hence if it is moreover unramified, it is trivial. Thus 
\[\mu\in R(\Psi) \Leftrightarrow \mu_{E_n}=1  \Leftrightarrow \mu_{E_n}(\Fr_{E_n})=1 \Leftrightarrow \mu(\Fr)^{nt/e}=1.\]
\end{proof}

As an immediate consequence we have:

\begin{corollary}\label{powers of nu in ramification group}
With the notations as in Corollary \ref{ramification group}, we obtain that $o(\Psi)=o(\nu^{r/e})=o(\nu_E^n)$.  
\end{corollary}

\subsection{Lattices and reduction modulo $\ell$}

An $\ell$-adic finite dimensional representation $\Theta$ of $\W_F$ is called \textit{integral} if there is a $\Zl[\W_F]$-stable lattice in $V$. In general when we say \textit{a lattice} in the space of a representation of $\W_F$, we shall always refer to a $\W_F$-stable lattice. We denote by $\Irr(\W_F,\Ql)_e$ the set of isomorphism classes of integral irreducible representations of 
$\W_F$. 

By \cite[Proposition 28.6]{BHbook}, if $\Theta\in \Irr(\W_F,\Ql)$, then there is an unramified character $\mu$ of $\W_F$ 
such that $\mu \Theta$ extends to an element of $\Irr_{\Ql}(\G_F)$. The group $\G_F$ being profinite, the representation $\mu\Theta$ is automatically integral, and we deduce that $\Theta$ is integral if and only if $\mu$ takes values in $\Zl^\times$. As explained 
in \cite[1.8]{Vproposal}, this happens if and only if $\det(\Theta)$ takes values in $\Zl^\times$, hence one deduces the well-known fact:

\begin{lemma}[{\cite[Lemma 1.9]{Vproposal}}] The representation~$\Theta\in 
\Irr(\W_F,\Ql)$ is integral if and only if $\det(\Theta)$ is integral.
\end{lemma}

Let~$\Theta\in \Irr(\W_F,\Ql)_e$ and choose a lattice~$L$ in~$V_{\Theta}$.   In this setting, we have a Brauer--Nesbitt principle:

\begin{lemma}[{\cite[1.10]{Vproposal}}]
The semi-simplification of the~$\Fl$-representation~$L/\m L$ is independent of the choice of~$L$
\end{lemma}

We let 
\[r_\ell(\Theta)=[L/\m L]^{ss}\]
denote this semisimplification.  If~$r_\ell(\Theta)$ is irreducible, we let \[\overline{\Theta}=r_{\ell}(\Theta),\] and say that~$\Theta$ \emph{lifts}~$\overline{\Theta}$.

Suppose now~$\Psi\in \Irr(\W_F,\Fl)$ (because we treat~$\ell$-adic and modular representations uniformly, in both cases we will often use~$\Psi$ for an element of $\Irr(\W_F,\Fl)$, but when needed for reduction modulo $\ell$, we will use $\Theta$ in the $\ell$-adic case as we did above). In this case, one can always lift~$\Psi$ to an irreducible representation in~$\Irr(\W_F,\Ql)_e$, this follows from the description of~$\Irr(\W_F,R)$ given in Theorem \ref{irreps of the Weil group}, and is explained in~\cite[2.6]{Vpartialcorrespondence}.    While not unique in general, when we wish to choose a lift of~$\Psi$ to~$\Irr(\W_F,\Ql)_e$ we will denote it by~$\tilde{\Psi}$.

On the other hand, it is not true that the reduction modulo $\ell$ of $\Theta\in \Irr(\W_F,\Ql)_e$ is irreducible in general, however it is well understood, and described 
in \cite[1.16]{Vproposal}. We first fix some notations. 
Take $\Theta\in \Irr(\W_F,\Ql)_e$ and let $\beta$ be an integral unramified character of $\W_F$ such that 
$\beta\Theta$ extends to $\G_F$. Write $\Psi$ under the form $\Ind_{\W_E}^{\W_F}(\Psi^{\tr}\otimes \tau)$ as in Theorem \ref{irreps of the Weil group}. Write
$\Theta^{\tr}$ under the form $\Ind_{\W_{E_n}}^{\W_E}(\chi)$ for some positive $n$ and some tamely ramified character $\chi$ of $\W_{E_n}$. 
Both $\Theta^{\tr}$ and $\chi$ are integral, and $\chi'=\beta_{E_n}\chi$ is a character of finite order, which can uniquely be written $\chi'=\alpha_{E_n} \chi'_\ell$ for 
$\alpha$ a $\Gal_{E}(E_d)$-regular character of $\W_{E_d}$ where $d$ a divisor of $n$, 
and $\chi'_\ell$ is of order a power of $\ell$. The representation $\Ind_{\W_{E_d}}^{\W_{E}}(\overline{\alpha})$ is an irreducible representation of 
$\W_E$, so 
\[\Psi(\alpha,\beta):=\overline{\beta_E}^{-1} \Ind_{\W_{E_d}}^{\W_{E}}(\overline{\alpha})= \Ind_{\W_{E_d}}^{\W_{E}}(\overline{\beta_{E_d}^{-1}\alpha})\] too. 
The representation $r_\ell(\tau)$ is also irreducible, and we thus write it $\overline{\tau}$. 
The quotient $m=n/d$ is either equal to $1$, or of the form $o(\nu_E^d)\ell^a$ (here $\nu$ is of course the $\ell$-modular absolute value) for an integer $a\geq 0$.

\begin{prop}\label{reduction of irreps}
Take $\Theta\in \Irr(\W_F,\Ql)_e$, fix the notations as above, then: 
\[r_\ell(\Theta)=\bigoplus_{k=0}^{m-1} \nu^{k} \Ind_{\W_{E}}^{\W_{F}}(\Psi(\alpha,\beta)\otimes \overline{\tau})\] 
where $ \Ind_{\W_{E}}^{\W_{F}}(\Psi(\alpha,\beta)\otimes \overline{\tau})$ hence all its twists are irreducible.
\end{prop}

According to Corollary, we can restate Proposition \ref{reduction of irreps} in the following less precise but simpler form.

\begin{prop}\label{reduction of irreps simpler}
Take $\Theta\in \Irr(\W_F,\Ql)_e$, then either $r_\ell(\Theta)$ is irreducible, or it is of the form 
\[r_\ell(\Theta)=\ell^a({\textstyle\bigoplus_{k=0}^{o(\Psi)-1}} \nu^k \Psi)\] for $a\geq 0$ and $\Psi \in \Irr(\W_F,\Fl)$.
\end{prop}
\begin{proof}
In Proposition \ref{reduction of irreps}, the integer $m$ is of the form $o(\nu_E^d)\ell^a$ when not equal to $1$, but 
$\Psi=Ind_{\W_{E}}^{\W_{F}}(\Psi(\alpha,\beta)\otimes \overline{\tau})$ satisfies $o(\Psi)=o(\nu_E^d)$ thanks to 
Corollary \ref{powers of nu in ramification group}. The result is now a consequence of $o(\Psi)$'s definition.
\end{proof}

We deduce the following corollary that we shall use later. 

\begin{corollary}\label{unramified in reduction}
Let $\Theta\in \Irr(\W_F,\Ql)_e$. Then the following are equivalent:
\begin{enumerate}[{(1)}] 
\item $r_\ell(\Theta)^{I_F}\neq 0$.
\item There is $n\geq 1$ equal to $1$ or otherwise of the form $n=o(\nu)\ell^r$ with $r\geq 0$, and $\chi$ an integral character of $\mathrm{W}_{E_n}$ 
with $\overline{\chi}=\mu_{E_n}$ for $\mu$ an unramified character of $\W_F$, such that 
\[\Theta= \Ind_{\mathrm{W}_{E_n}}^{\mathrm{W}_F}(\chi).\]  
\end{enumerate}
When they are satisfied \[{r_\ell(\Theta)}^{I_F}=r_\ell(\Theta)={\textstyle\bigoplus_{k=0}^{n-1}} \nu^k \mu=\ell^r 
({\textstyle\bigoplus_{k=0}^{o(\nu)-1}} \nu^k \mu).\]
\end{corollary}
\begin{proof}
If $\Theta= \Ind_{\mathrm{W}_{E_n}}^{\mathrm{W}_F}(\chi)$ with $\chi$ as in the statement, then 
\[{r_\ell(\Theta)}^{I_F}=r_\ell(\Theta)=\sum_{k=0}^{n-1} \nu^k \mu\] thanks to Proposition \ref{reduction of irreps}. Conversely, suppose that $r_\ell(\Theta)^{I_F}\neq 0$. Write $\Ind_{\W_E}^{\W_F}(\Theta^{\tr}\otimes \tau)$ and $\Theta^{\tr}=\Ind_{\W_{E_n}}^{\W_E}(\chi)$ as discussed before Proposition \ref{reduction of irreps}. Thanks to [ibid.], we have \[r_\ell(\Theta)=\bigoplus_{k=0}^{m-1} \nu^{k} \Ind_{\W_{E}}^{\W_{F}}(\Psi(\alpha,\beta)\otimes \overline{\tau})\] with all representations $\Ind_{\W_{E}}^{\W_{F}}(\Psi(\alpha,\beta)\otimes \overline{\tau})$ irreducible. However one of them has nonzero $I_F$-fixed subspace by assumption, so they must all be unramified characters. The fact that they are characters implies that $F=E=E_d$, i.e. $d=1$, and that 
$\chi=(\beta^{-1}\alpha)_{E_n}\chi'_\ell$, so that $\overline{\chi}=(\overline{\beta^{-1}\alpha})_{E_n}$. Moreover setting 
$\mu=\overline{\beta^{-1}\alpha}$, the character $\mu_{E_n}$ must be unramified, hence $\mu$ as well, as $I_{E_n}=I_F$. This proves the converse implication.
\end{proof}

\subsection{Tensor products of modular representations of $\W_F$}\label{tensor for W_F}

We now study~$\Psi\otimes \Psi'$ when~$\Psi$ and~$\Psi'$ belong to~$\Irr(\W_F,\Fl)$. First we observe some fundamental differences with the case of characteristic zero, where such a tensor 
product is semisimple: this will have consequences later in the paper, we will have to define a semisimple tensor product for $\W_F$-semisimple Deligne representations. In fact when~$R=\Ql$, 
by \cite[Proposition 28.7]{BHbook}, a finite dimensional representation~$\Phi$ of~$\W_F$ is semisimple if and only if it is \textit{Frobenius semisimple}, i.e. if and only if 
$\Phi(\Fr)$ is semisimple. As an immediate consequence, the tensor product of two semisimple~$\Ql$-representations of~$\W_F$ is itself semisimple. In the modular case the situation is totally different: in general, Frobenius semisimple does not imply semisimple, neither does semisimple imply Frobenius semisimple, and the tensor product does 
not preserve semisimplicity. We give a list of examples illustrating these differences, we use the following lemma:
 
\begin{lemma}\label{uniserial example}
Let~$E/F$ be a quadratic extension and~$\ell=2$.  The~$\Fl$-representation~$\Ind_{\W_E}^{\W_F}(\1_{W_E})$ is uniserial of length~$2$, with both subquotients isomorphic to~$\1_{\W_F}$.
\end{lemma}
\begin{proof}
It is two-dimensional, and contains~$\1_{\W_F}$ as a submodule with multiplicity one by Frobenius reciprocity law. 
The quotient~$\Ind_{\W_E}^{\W_F}(\1_{\W_E})/\1_{\W_F}$ is a character~$\mu$, so~$\mu^{-1}$ is a submodule of~$\Ind_{\W_E}^{\W_F}(\1)$ which is self-dual.   Hence by Frobenius reciprocity again~$\mu_E=\1$ so~$\mu^2=\1$, but the condition~$\ell=2$ implies that~$\mu$ is then trivial. All in all,~$\Ind_{\W_E}^{\W_F}(\1_{\W_E})$ is indecomposable with head and socle~$\1_{\W_F}$.
\end{proof}
Of course, one can concoct similar examples for any prime~$\ell$ using degree~$\ell$ extensions.

\begin{ex}\label{counter examples}
For these examples we take~$\ell=2$.  In all our examples we choose a separable quadratic extension~$E/F$ with Galois involution~$\sigma$; a character~$\chi:\W_E\rightarrow\overline{\mathbb{F}_2}^\times$, and consider the~$2$-dimensional representation
\[\Phi=\Ind_{\W_E}^{\W_F}(\chi),\]
the~$\overline{\mathbb{F}_2}$-representation induced from~$\chi$.  We fix a geometric Frobenius element~$\Fr_E$ in~$\W_E$.

By the Artin reciprocity isomorphism of local class field theory, we can identify the topologically abelianized quotient~$\W_E^{\ab}$ of~$\W_E$ with the group~$E^\times$, and similarly~$\W_F^{ab}$ with~$F^\times$.  Via this isomorphism the Frobenius elements~$\Fr,\Fr_E$ corresponds to uniformizers~$\w,\w_E$ of~$F$ and~$E$ respectively.  We also recall that the restriction functor from~$\W_F^{\ab}$ to~$\W_E^{\ab}$ (relative to the natural inclusion~$\W_E^{ab}$ into~$\W_F^{ab}$) corresponds to the norm~$N_{E/F}:E^\times\rightarrow F^\times$.  In the second example, we use this identification to from characters of~$\W_F$ and~$\W_E$ to characters of~$F^\times$ and~$E^\times$ respectively.
\begin{enumerate}[{(1)}]
\item  \textit{Frobenius semisimple does not imply semisimple}:  We take~$E/F$ ramified and~$\chi=\1_{\W_E}$. Then~$\Fr_E=\Fr$ acts trivially on~$\Phi$ as~$\Phi_E=\1_{\W_E}\oplus \1_{\W_E}$, however~$\Phi$ is not semisimple according to  Lemma \ref{uniserial example}.

 \item \textit{Semisimple does not imply Frobenius semisimple}: 
 take~$F=\Q_2$ and~$E=F_2$ the unique (up to isomorphism) quadratic unramified extension.  Write 
\[E^\times\simeq \Z\times (\Z/3\Z) \times (1+2\Z_2),\] and, by abuse of notation, we identify~$E^\times$ with this product.   Notice that the norm~$N_{k_E/k_F}$ is trivial as a character of 
$k_E^\times$ 
because~$k_F^\times=\{1\}$, in particular \[N_{E/F}(\Z/3\Z)=\{1\}.\] We take~$\chi$ a non trivial character of~$\Z/3\Z$ with values in~$\overline{\F_2}$ and extend it trivially on~$\Z$ and extend it on~$(1+2\Z_2)$ to a character~$\chi$ of~$E^\times$. Then~$\chi$ does not factor through~$N_{E/F}$, but it satisfies~$\chi(\w_E)=1$. 
The corresponding character~$\chi$ of~$\W_E$ is thus~$\Gal(E/F)$-regular, but satisfies~$\chi(\Fr_E)=1$. Hence 
the representation~$\Ind_{\W_E}^{\W_F}(\chi)$ is an irreducible representation of~$\W_F$. If~$\Fr$ acted via a scalar on this representation, then~$\Ind_{\W_E}^{\W_F}(\chi)$ would be an irreducible representation of~$I_F=I_E$ ($E/F$ is unramified). This is absurd as its restriction to~$\W_E$ is a direct sum of two lines. 
However~$\Fr_E$ acts trivially on~$\Ind_{\W_E}^{\W_F}(\chi)_{E}=\chi\oplus \chi^\sigma$, and moreover one can take~$\Fr_E=\Fr^2$. 

We conclude that~$\Phi$ is irreducible hence semisimple,~$\Phi(\Fr)$ is not a scalar hence not~$\Id$, but~$\Phi(\Fr)^2=\Id$, hence~$\Phi(\Fr)$ is not semisimple because~$\ell=2$.

\item\label{Part3}\label{otimescounter}  \textit{$\otimes$ does not preserve semi-simplicity}: 
take~$\chi$ to be~$\Gal(E/F)$-regular, hence~$\Phi$ and~$\Phi^\vee$ are irreducible. 
The following decomposition of their tensor product, which follows from Mackey theory, 
\[ \Phi \otimes \Phi^\vee=\Ind_{\W_E}^{\W_F}(\1_{\W_E}) \oplus \Ind_{\W_E}^{\W_F}(\chi^{-1}\chi^{\sigma})\] 
and Lemma \ref{uniserial example} show that it is not semisimple.

\end{enumerate}
\end{ex}

\begin{rem}\label{SerreRemark} By the main result of \cite{Serre-tensor}, the tensor product of two semisimple~$\Fl$-representations~$\Phi$ and~$\Phi'$ of respective dimensions~$d$ and~$d'$ is semisimple if~$d+d'-2<\ell$. Example \ref{counter examples} (\ref{Part3}) above fortunately falls outside the range of Serre's result!
\end{rem}

If~$\Phi$ and~$\Phi'$ are two semisimple representations of~$\W_F$, we denote by \[\Phi\otimes_{\ss}\Phi'=(\Phi\otimes \Phi')_{\ss}\] 
the semisimplifaction of their tensor product.  

To end this section, we want to understand~$(\Psi\otimes \Psi')^{I_F}$ when~$\Psi$ and~$\Psi'$ are two irreducible  banal~$\Fl$-representations of~$\W_F$. It is the central result of this section and will be used later.

\begin{prop}\label{banal tensor unramified reduction}
Let~$\Psi$ and~$\Psi'$ be two banal irreducible representations of~$\W_F$, and let 
$\tilde{\Psi}$ and~$\tilde{\Psi}'$ be two~$\ell$-adic lifts of~$\Psi$ and~$\Psi'$. Choose 
$L$ and~$L'$ lattices in~$V_{\tilde{\Psi}}$ and~$V_{\tilde{\Psi}'}$ so that \[\Psi\simeq L\otimes_{\Zl} \Fl\] and 
\[\Psi'\simeq L'\otimes_{\Zl} \Fl.\] Set~$M=L\otimes L'$ and for any~$\Ql$-subspace~$W$ of~$V_{\tilde{\Psi}}\otimes V_{\tilde{\Psi}'}$, set~$\overline{W}=(W\cap M)\otimes_{\Zl} \Fl$. Then  
 \[\overline{(V_{\tilde{\Psi}}\otimes V_{\tilde{\Psi}'})^{I_F}}=(V_\Psi\otimes V_{\Psi'})^{I_F}\] and 
 \[((V_\Psi\otimes V_{\Psi'})^{I_F})_{\ss}\simeq (V_\Psi\otimes_{\ss} V_{\Psi'})^{I_F}.\]
\end{prop}
\begin{proof}
 We write 
\[\tilde{\Psi}\otimes \tilde{\Psi}'=\bigoplus_{i}\mu_i\oplus\bigoplus_k \Psi_k\] where~$\mu_i$ are unramified characters, and 
$\Psi_k$ the other ramified irreducible representations of~$\W_F$ occurring. In particular,
\[(\Phi\otimes \Phi')^{I_F}=\bigoplus_{i}\mu_i.\] 

The inclusion 
\[\overline{(V_{\tilde{\Psi}}\otimes V_{\tilde{\Psi}'})^{I_F}}\subset (V_\Psi\otimes V_{\Psi'})^{I_F}\] is clear. Suppose that the inclusion was strict, this would imply that one of the~$\Psi_k$'s is such that 
$(\overline{V_{\Psi_k}})^{I_F}$ is nonzero. For the sake of contradiction we suppose that it is the case. According to Corollary \ref{unramified in reduction}, setting either~$E=F$ or~$E=F_{o(\nu)\ell^m}$ for~$m\geq 0$: 
\[\Psi_k=\Ind_{\mathrm{W}_E}^{\mathrm{W}_F}(\chi)\] for~$\chi$ an integral~$\Gal_F(E)$-regular character of~$E^*$ such that 
$r_\ell(\chi)=\mu_E$ with~$\mu$ an unramified character of~$\W_F$.

Now by definition, the space~$\Hom_{\mathrm{W}_F}(\tilde{\Psi}\otimes \tilde{\Psi}',\Ind_{\mathrm{W}_E}^{\mathrm{W}_F}(\chi))$ is nonzero. 
Take~$U$ a nonzero element in this space, as~$\Ind_{\mathrm{W}_E}^{\mathrm{W}_F}(\chi)$ is irreducible~$U$ is surjective. 
By Mackey theory \[\Ind_{\mathrm{W}_E}^{\mathrm{W}_F}(\chi)\vert_{\mathrm{W}_E}=\bigoplus_{s\in \mathrm{W}_F/\mathrm{W}_E} s.D\] where 
$D$ is a line over~$\Ql$ on which~$\mathrm{W}_E$ acts as~$\chi$, and~$\mathrm{W}_E$ more generally acts on~$s.D$ as~$\chi^s$. Denote by~$p$ the projection on~$D$ with respect to 
$\bigoplus_{s\neq 1} s.D$, then~$p(U(M))=O_D$ is a lattice in~$D$ by \cite[9.3, (vi)]{Vbook}, and \[N=\bigoplus_{s\in \mathrm{W}_F/\mathrm{W}_E} s.O_D=\Ind_{\mathrm{W}_E}^{\mathrm{W}_F}(O_D)\]
 is a~$\W_F$-stable lattice in~$\Ind_{\mathrm{W}_E}^{\mathrm{W}_F}(\chi)$. Notice that for~$s\in \mathrm{W}_F/\mathrm{W}_E$, the map 
~$p_s= s\circ p \circ s^{-1}$ the projection on~$D$ with respect to~$\bigoplus_{s'\neq s} s'.D$. We have 
\[p_s(U(M))=s(p(U(s^{-1}.M)))=s(p(U(M)))=s.O_D,\] hence \[U(M)=\bigoplus_{s\in \mathrm{W}_F/\mathrm{W}_E} p_s(U(M))=\bigoplus_{s\in \mathrm{W}_F/\mathrm{W}_E} 
 s.O_D=N.\]
 
Tensoring by~$\Fl$, we obtain a surjective~$\W_F$-intertwining operator~$\overline{U}$ from 
$\Psi\otimes \Psi'$ onto \[\Ind_{\mathrm{W}_E}^{\mathrm{W}_F}(r_\ell(\chi))=\Ind_{\mathrm{W}_E}^{\mathrm{W}_F}(\mu_E),\] i.e. 
$\Ind_{\mathrm{W}_E}^{\mathrm{W}_F}(\mu_E)$ is a quotient of~$\Psi\otimes \Psi'$. Dualizing, we obtain that 
\[\Ind_{\mathrm{W}_E}^{\mathrm{W}_F}(\mu_E^{-1})\] 
 is a submodule of~$\Psi^\vee\otimes {\Psi'}^\vee$. 
By Frobenius reciprocity, we also have for any~$k\in \Z$:
\[\Hom_{\mathrm{W}_F}(\mu^{-1}\nu^{-k},\Ind_{\mathrm{W}_E}^{\mathrm{W}_F}(\mu_E^{-1}))\simeq 
\Hom_{\mathrm{W}_E}((\mu^{-1}\nu^{-k})_E,\mu_E^{-1})\]
\[\simeq \Hom_{\mathrm{W}_E}(\mu_E^{-1},\mu_E^{-1})\simeq \Fl.\] Hence 
for any~$k$,~$\mu^{-1}\nu^{-k}$ is a submodule of~$\Ind_{\mathrm{W}_E}^{\mathrm{W}_F}(\mu_E^{-1}))$ which is itself 
a submodule of~$\Psi^\vee\otimes {\Psi'}^\vee$. Dualizing again, we deduce that~$\mu\nu^{k}$ is a quotient of~$\Psi\otimes \Psi'$ 
for all~$k$. However this would imply that~${\Psi'}^\vee\simeq \mu\nu^{k}\Psi$ for all~$k$, and this would in particular imply that~$\nu\Psi\simeq \Psi$ which is absurd. Hence we just proved that 
 \[\overline{(V_{\tilde{\Psi}}\otimes V_{\tilde{\Psi}'})^{I_F}}=(V_\Psi\otimes V_{\Psi'})^{I_F}.\] In fact we proved that~$\overline{\Psi_k}^{I_F}=\{0\}$ for all~$\Psi_k$. However~$r_\ell(\Psi_k)^{I_F}=\{0\}$ if and only if~$\overline{\Psi_k}^{I_F}=\{0\}$ according to Corollary \ref{unramified in reduction}, hence the isomorphism \[[(V_\Psi\otimes V_{\Psi'})^{I_F}]_{\ss}\simeq (V_\Psi\otimes_{\ss} V_{\Psi'})^{I_F}.\]
\end{proof}

\section{Deligne representations}\label{motivation}

Here we classify in terms of irreducible representations of $\W_F$ what we call~$\W_F$-semisimple Deligne representations (see Definition \ref{def deligne rep}) up to a certain equivalence relation (Definition \ref{equivalence}). The purpose of doing this is that in the modular case, we will parametrize (see Section \ref{section C corresp}) irreducible representations of $\GL(n,F)$ by $n$-dimensional equivalence classes of semisimple Deligne representations.

\subsection{Definitions, notations and basic properties}

Here is our definition of Deligne representations, which as we shall see, specializes to the usual definition when $R=\Ql$.

\begin{definition}\label{def deligne rep}
\begin{itemize}
\item  A \emph{Deligne~representation} of~$\W_F$ is a pair~$(\Phi,U)$ where~$\Phi$ is a finite dimensional representation of~$\W_F$, and~$U\in\Hom_{\W_F}(\nu\Phi,\Phi)$. 
\item We say that $(\Phi,U)$ is a \textit{$\W_F$-semisimple Deligne representation} of $\W_F$ if 
$\Phi$ is semisimple as a representation of $\W_F$.
\end{itemize}
For~$(\Phi,U),(\Phi',U')$ Deligne representations of~$\W_F$, we write
\[\Hom_{\WD}((\Phi,U),(\Phi',U'))=\{f\in\Hom_{\mathrm{W}_F}(\Phi,\Phi'): f\circ U=U'\circ f\},\]
\[\End_{\WD}(\Phi,U)=\Hom_{\WD}((\Phi,U),(\Phi,U)),\] and 
\[\Iso_{\WD}((\Phi,U),(\Phi',U'))=\Hom_{\WD}((\Phi,U),(\Phi,U))\cap \Iso_{\W_F}(\Phi, \Phi').\]
If this latter space is non empty, we say that $(\Phi,U)$ and $(\Phi,U')$ are isomorphic.
\end{definition}

\begin{notation}
We will sometimes write $V_{(\Phi,U)}$ for $V_{\Phi}$, and also write $d_{(\Phi,U)}$ or $d_\Phi$ for 
$\dim_R(V_\Phi)$.
\end{notation}

For $U$ and endomorphism of a finite dimensional~$R$-vector space, we denote by~$\Sp(U)$ the set of its eigenvalues.

\begin{rem}\label{U=N in char zero}
\begin{enumerate}[{(1)}]
\item Take $\Phi$ a finite dimensional representation of~$\W_F$, and $U\in \End_{W_F}(V_\Phi)$ with Jordan decomposition 
$D+N$ ($D$ semismple and $N$ nilpotent). Then $U\in\Hom_{W_F}(\nu\Phi,\Phi)$ if and only if $D$ and $N$ are in $\Hom_{W_F}(\nu\Phi,\Phi)$. 
\item\label{U=N in char zero2} If $R=\Ql$, and $U\in\Hom_{W_F}(\nu\Phi,\Phi)$, then $U=N$ as $\mathrm{Spec}(U)$ is stable under multiplication by $q$.
\end{enumerate}
\end{rem}

\begin{notation}
The Jordan decomposition of $U$ will always be denoted by $D+N$, unless explicitly stated.
\end{notation}

As we only consider Deligne representations of~$\W_F$, we will suppress the ``of~$\W_F$'' in our notation. 
The direct sum of two Deligne representations $(\Phi,U)$ and $(\Phi',U')$ is defined as \[(\Phi,U)\oplus(\Phi',U')=(\Phi\oplus\Phi',U\oplus U'),\] notice that it preserves $\W_F$-semi-simplicity.
Let 
us introduce some further notations.

\begin{notation}
We introduce the following notation:
\begin{itemize}
\item~$\Rep_{\ss}(\WD, R)$ for the set of isomorphism classes of~$\W_F$-semisimple Deligne~$R$-representations.
\item~$\Indec_{\ss}(\WD, R)$ for the Deligne~$R$-representations in~$\Rep_{\ss}(\WD, R)$ which are indecomposable under direct sum of Deligne~$R$-representations.
\item~$\Irr_{\ss}(\WD, R)$ for the irreducible Deligne~$R$-representations in~$\Rep_{\ss}(\WD, R)$.\newline(Note that~$(\Phi,U)\in\Irr_{\ss}(\WD, R)$ does not imply~$\Phi\in\Irr(W_F,R)$.)
\item~$\Nilp_{\ss}(\WD, R)$ for the Deligne~$R$-representations~$(\Phi,U)\in\Rep_{\ss}(\WD, R)$ with~$U=N$ nilpotent. In particular, by Remark \ref{U=N in char zero} (\ref{U=N in char zero2}),
\[\Rep_{\ss}(\WD, \Ql)=\Nilp_{\ss}(\WD, \Ql)\] 
\end{itemize}
\end{notation}

For $(\Phi,U)$ and $(\Phi',U')$ in $\Rep_{\ss}(\WD, R)$, then 
 is also in $\Rep_{\ss}(\WD, R)$.

\begin{definition}\label{dualdef}
The \emph{dual} of a Deligne~$R$-representation~$(\Phi,U)$ with~$U=D+N$ is defined by 
\[(\Phi,U)^\vee=(\Phi^\vee,D^\vee-N^\vee).\] 
\end{definition}

Clearly,~$\Rep_{\ss}(\WD, R)$ is stable under direct sum and duals:
\begin{lemma}
For~$(\Phi,U),(\Phi',U')\in\Rep_{\ss}(\WD, R)$, then \[(\Phi,U)\oplus(\Phi',U'),(\Phi,U)^\vee\in \Rep_{\ss}(\WD, R).\]
\end{lemma}

The \emph{tensor product} of Deligne~$R$-representations is defined by the formula 
\[(\Phi,U)\otimes (\Phi',U')=(\Phi\otimes \Phi',U \otimes \Id\oplus \Id\otimes U').\]

The set~$\Rep_{\ss}(\WD, R)$ is not stable under this operation, according to Example \ref{counter examples}, \ref{otimescounter}.   We will introduce a \emph{semisimple} tensor product~$\otimes_{\ss}$ in Section \ref{tensor section}.  Of course, whenever~the tensor product is~$\W_F$-semisimple we will have~$\otimes_{\ss}=\otimes$, which will be the case when the characteristic is not too small in relation to the dimensions of the representations by Remark \ref{SerreRemark}.

Now we introduce an equivalence relation~$\sim$ on~$\Rep_{\ss}(\WD, R)$. We do this as we shall later parametrize irreducible representations of the group~$\GL_n(F)$ by equivalence classes rather than isomorphism classes of~Deligne~$R$-representations.

\begin{definition}\label{equivalence}
The definition is in two steps:
\begin{enumerate}[{(1)}]
\item  Deligne~$R$-representations~$(\Phi,U), (\Phi',U')\in \Indec_{\ss}(\WD, R)$ are \emph{equivalent}, denoted~$(\Phi,U)\sim (\Phi',U')$, if there exists~$\l\in R^\times$ such that 
\[(\Phi',U')\simeq (\Phi,\l U).\]
\item\label{equivalence2} In the general case,~$(\Phi,U),(\Phi',U')\in \Rep_{\ss}(\WD, R)$ are \emph{equivalent}, denoted~$(\Phi,U)\sim (\Phi',U')$, if one can decompose~$(\Phi',U')=\bigoplus_{i=1}^r (\Phi_i,U_i)$ and~$(\Phi,U)=\bigoplus_{i=1}^r (\Phi_i,U_i)$ such that~$(\Phi_i,U_i)\sim (\Phi_i,U_i)$ in~$\Indec_{\ss}(\WD, R)$.
\end{enumerate}
\end{definition}

\begin{rem}\label{remark Krull}
To see that Definition \ref{equivalence} defines an equivalence relation we use that the decomposition of a Deligne~$R$-representation into indecomposable Deligne~$R$-representations is unique. This will be a consequence of our classification of indecomposable~$\W_F$-semisimple Deligne representations: we show as a consequence of Propositions \ref{Sp are indec} and \ref{construction indec}, and Theorems \ref{V_0} that the endomorphism ring of such a representation is local, which is what is needed in the proof of the Krull--Schmidt theorem.
\end{rem}

\begin{notation}We introduce square brackets to denote equivalence classes:
\begin{itemize}
\item For~$(\Phi,U)\in \Rep_{\ss}(\WD, R)$, we let~$[\Phi,U]$ denote its equivalence class.
\item~$[\Rep_{\ss}(\WD, R)]=\Rep_{\ss}(\WD, R)/\sim$.
\item~$[\Irr_{\ss}(\WD, R)]=\Irr_{\ss}(\WD, R)/\sim$.
\item~$[\Indec_{\ss}(\WD, R)]=\Indec_{\ss}(\WD, R)/\sim$.
\item~$[\Nilp_{\ss}(\WD, R)]=\Nilp_{\ss}(\WD, R)/\sim$.
\end{itemize}
\end{notation}

Let~$[\Phi,U],[\Phi',U']\in\Rep_{\ss}(\WD, R)$, for any choice of representatives their direct sums and duals are equivalent, so we let
\begin{align*}
[\Phi,U]\oplus [\Phi',U']&=[(\Phi,U)\oplus (\Phi',U')]\\
[\Phi,U]^\vee&=[(\Phi,U)^\vee],
\end{align*}
giving a well defined direct sum and dual on~$[\Rep_{\ss}(\WD, R)]$. 

In the~$\ell$-adic case, we recall that~$\Rep_{\ss}(\WD, \Ql)=\Nilp_{\ss}(\WD, \Ql)$, hence the following proposition shows that in this case one gains nothing by introducing the equivalence relation~$\sim$.

\begin{prop}\label{iso vs equiv nilp}
The canonical surjection from~$\Nilp_{\ss}(\WD, R)$ onto~$[\Nilp_{\ss}(\WD, R)]$ is the identity, we write: 
\[[\Nilp_{\ss}(\WD, R)]=\Nilp_{\ss}(\WD, R).\]
\end{prop}
\begin{proof}
Take~$(\Phi,N)$ a Deligne representation with~$N$ nilpotent, it is enough to prove 
that~$(\Phi,N)$ and~$(\Phi,\l N)$ are isomorphic if~$\l\in R^\times$. Let~$r$ be the 
nilpotency index of~$N$. As the iterated kernels~$\Ker(N^k)$ are all~$\W_F$-stable, we can construct 
a~$\W_F$-stable complement~$S_{r-1}$ of~$\Ker(N^{r-1})$ in~$\Ker(N^r)=V_{\Phi}$. Then~$N(S_{r-1})$ is 
also~$\W_F$-stable, and \[N(S_{r-1})\cap \Ker(N^{r-2})=\{0\}.\] Hence~$N(S_{r-1})\oplus \Ker(N^{r-2})$ admits 
a~$\W_F$-stable complement~$U_{r-1}$ in~$\Ker(N^{r-1})$, and \[S_{r-2}=N(S_{r-1})\oplus U_{r-1}\] is 
a~$\W_F$-stable complement of~$\Ker(N^{r-2})$ in~$\Ker(N^{r-1})$ such that~$N(S_{r-1})\subset S_{r-2}$. 
Continuing in this fashion, we obtain for~$k=0,\dots, r-1$, a~$\W_F$-stable complement~$S_k$ of~$\Ker(N^k)$ 
in~$\Ker(N^{k+1})$ such that~$N(S_{k})\subset S_{k-1}$ for~$k\geq 1$, and \[V_\Phi=\bigoplus_{k=0}^{r-1} S_k.\]
We define~$P\in \GL(V_{\Phi})$ to be equal to~$\l^{i} \mathrm{Id}_{S_i}$ on~$S_i$. Then~$P$ commutes with~$\Phi(w)$ for all~$w\in \mathrm{W}_F$, and as is checked on each~$S_i$, one has 
$P\l N= NP$.  Hence~$P$ defines the required isomorphism.
\end{proof}

\begin{corollary}
We have~$[\Rep_{\ss}(\WD, \Ql)]=\Nilp_{\ss}(\WD, \Ql)$.
\end{corollary}

In particular, for~$[\Phi,U],[\Phi',U']$ in~$\Rep_{\ss}(\WD, \Ql)$, the equivalence class~$[(\Phi,U)\otimes (\Phi',U')]$ is independent of the choice of representatives for~$[\Phi,U],[\Phi',U']$.   This is not true anymore when~$R=\Fl$, first we already saw in example \ref{counter examples} that~$\W_F$-simplicity is not preserved, but even if it is,~$[\Phi,U]\otimes [\Phi',U']$ is still not well-defined in general. However there is a also a natural solution to this problem as we shall see in Section \ref{tensor section}. We shall indeed define a tensor product on~$[\Rep_{\ss}(\WD, \Fl)]$ which is associative and distributive with respect to~$\oplus$ on the left and on the right. We give an example:

\begin{ex}\label{non equiv tensor}
Take~$\ell\neq 2$,~$\Phi=\1\in \Irr(\W_F,\Fl)$ and suppose that~$q\equiv 1[\ell]$, then 
\[(\1,\Id)\otimes (\1,\Id)=(\1,2\Id),~\text{whereas}~(\1,\Id)\otimes (\1,-\Id)=(\1,0),\] which are inequivalent. 
However, for all pairs~$(\l,\mu)\in (\Fl^\times)^2$ outside the hyperplane of~$(\Fl)^2$ defined by~$\mu+\l=0$, we have 
\[[(\1,\l\Id)\otimes (\1,\mu\Id)]=[1,\Id]\] and we shall set \[[\1,\Id]\otimes [\1,\Id]=[\1,\Id].\]
\end{ex}


However, there are two cases where the tensor product behaves well which will prove fundamental later:

\begin{lemma}\label{good case tensor}
Let~$(\Phi,U),(\Phi',U')\in \Rep_{\ss}(\WD, \Fl)$.
\begin{enumerate}[{(1)}]
 \item\label{goodtensor1} If~$\Phi$ is a direct sum of characters, then~$(\Phi,U)\otimes (\Phi',U')$ is~$\W_F$-semisimple.
 \item\label{goodtensor2} If moreover~$U=N$ is nilpotent, and if~$(\Phi_1,U_1)\sim (\Phi,U)$ and~$(\Phi_1',U_1')\sim (\Phi',U')$, 
 then \[(\Phi_1,U_1)\otimes (\Phi_1',U_1')\sim (\Phi,U)\otimes (\Phi',U').\]
\end{enumerate}
\end{lemma}
\begin{proof}
We leave (\ref{goodtensor1}) to the reader. For (\ref{goodtensor2}), it is enough to treat the case where~$(\Phi',U')$ is indecomposable.  Take~$(\Phi_1,U_1)\sim (\Phi,U)$ and~$(\Phi_1',U_1')\sim (\Phi',U')$, then~$(\Phi_1,U_1)=(\Phi,U)$ as~$U$ is nilpotent, and~$(\Phi_1',U_1')=(\Phi',\l U')$. This implies
\[(\Phi_1,U_1)\otimes (\Phi_1', U_1')=(\Phi,U)\otimes (\Phi',\l U') \sim (\Phi,\l^{-1}U)\otimes (\Phi',U')= (\Phi,U)\otimes (\Phi',U')\]
because~$U$ is nilpotent.
\end{proof}

In the situation of the lemma we define
\[[\Phi,U]\otimes [\Phi',U']=[(\Phi,U)\otimes(\Phi',U')].\]

We introduce the following definition, which will be convenient when dealing with indecomposable objects in 
$\Rep_{\ss}(\WD, R)$.

\begin{definition}
Take~$\Psi\in \Irr(\W_F,R)$, we say that~$(\Phi,U)\in \Rep_{\ss}(\WD, R)$ is supported on~$\Z_\Psi$ if 
$\Phi=\bigoplus_i \Psi_i$, for~$\Psi_i\in \Z_\Psi$.
\end{definition}

Take~$(\Phi,U)\in \Rep_{\ss}(\WD, R)$ and~$\Psi\in \Irr(\W_F,R)$. We write~$\Phi(i,\Psi)$ for its~$\nu^i\Psi$ 
isotypic component which only depends on the class~$i\pmod{o(\Psi)}$.   As~$U$ belongs to~$\Hom_{\W_F}(\nu\Phi,\Phi)$, the~$\W_F$-subrepresentation~$\Phi(\Z_\Psi):=\bigoplus_{i=1}^{o(\Psi)-1} \Phi(i,\Psi)$ is stable under~$U$.   As there is a finite subset~$S$ of~$\Irr(\W_F,R)$ such that~$\Phi=\bigoplus_{\Psi\in S} \Phi(\Z_\Psi)$, we deduce the following lemma:

\begin{lemma}\label{indec supported on a line}
If~$(\Phi,U)\in \Indec_{\ss}(\WD, R)$, then there is~$\Psi\in \Irr(\W_F,R)$ such that~$(\Phi,U)$ is supported on 
$\Z_\Psi$.
\end{lemma}

\subsection{Irreducible~$\W_F$-semisimple Deligne representations}

We start with the following simple observation.  

\begin{lemma}
If~$(\Phi,U)\in \Irr_{\ss}(\WD, R)$, then either~$U$ is bijective, or~$U$ is zero.
\end{lemma}
\begin{proof}
$\Ker(U)$ is a~$(\Phi,U)$-stable subspace.
\end{proof}

An irreducible Deligne representation~$(\Phi,0)$ is nothing more than an irreducible representation of~$\W_F$. In 
particular, when~$R=\Ql$, the map~$\Phi\mapsto (\Phi,0)$ is a bijection between 
$\Irr(\W_F,\Ql)$ and~$\Irr_{\ss}(\WD, \Ql)$.  When~$R=\Fl$, we record this as a lemma.

\begin{lemma}\label{irred nilp}
The map~$\Phi\mapsto (\Phi,0)$ is a bijection between~$\Irr(\W_F,\Fl)$ and~$\Irr_{\ss}(\WD, \Fl)\cap \Nilp_{\ss}(\WD, \Fl)$.
\end{lemma}

\begin{notation}
By abuse of notation, we set~$\Psi=(\Psi,0)=[\Psi,0]$.
\end{notation}

We now consider the case~$U$ bijective, hence~$R=\Fl$. We first start with an example. We let 
$\Psi\in \Irr(\W_F,\Fl)$ and~$I \in\Iso_{\mathrm{W}_F}(\nu^{o(\Psi)}\Psi,\Psi)$.

\begin{lemma}\label{irreducible}
 Let~$C(\Psi,I)=(\Phi(\Psi),C_I)$ be the Deligne~$\Fl$-representation defined by
\begin{align*}
\Phi(\Psi)&=\bigoplus_{k=0}^{o(\Psi)-1} \nu^{k}\Psi\\ 
C_I(x_0,\dots,x_{o(\Psi)-1})&=( I(x_{o(\Psi)-1}),x_0,\dots,x_{o(\Psi)-2}).
\end{align*}
Then~$C(\Psi,I) \in\Irr_{\ss}(\WD, \Fl)$.
 \end{lemma}
\begin{proof} 
The endomorphism~$C_I$ belongs to~$\Hom_{\W_F}(\nu\Phi,\Phi)$ by definition. Whenever~$W$ is a~$\W_F$-stable subspace of~$V:=V_\Phi$, we set~$W(i)=W(i,\Psi)$. Let~$W$ be a nonzero Deligne subrepresentation of~$V$. As it is~$\W_F$-stable, we have \[W=\bigoplus_{i=0}^{o(\Psi)-1} W(i),\] and as~$W\neq 0$, there is an~$i\in\{0, \ldots,o({\Psi})-1\}$ such that~$W(i)\neq 0$. Hence~$W(i)=V(i)=\nu^{i}\Psi$ by irreducibility of~$V(i)$. But then~$C_I^k(W(i))=C_I^k(V(i))=V(i+k[o(\Psi)])\subset W$ for all~$k\in \Z$, and~$W=V$.
\end{proof}

\begin{lemma}\label{iso on the same line}
With~$\Psi$ and~$I$ as above,~$C(\Psi,I)=C(\nu^k\Psi,I)$ for all~$k\in \Z$.
\end{lemma}
\begin{proof}
By definition~$C_I\in \Iso_{\W_F}(\Phi(\Psi),\nu^{-1}\Phi(\Psi))$ and commutes with~$C_I$.
\end{proof}

\begin{notation}
It thus makes sense to set \[C(\Z_\Psi,I):=C(\Psi,I).\]
\end{notation}

In fact, the dependence on~$I$ disappears when one considers the equivalence class:

\begin{lemma}\label{equiv irred}
With notations as in Lemma \ref{irreducible}, the equivalence class~$[C(\Z_\Psi,I)]$ is independent of 
$I\in \Hom_{\W_F}(\nu\Psi,\Psi)$.
\end{lemma}
\begin{proof}
Write~$C(\Psi,I)=(\Phi(\Psi),C_I)$ and take~$\l\in \Fl^\times$. Set~$V=V_{\Phi}$. Take~$u\in \Fl$ such that~$u^{o(\Psi)}=\l$. Define~$A\in \End(V)$ by the formula 
\[A(x_0,x_1,x_2,\dots,x_{o(\Psi-1)})=(x_0,u x_1,\dots,u^{o(\Psi)-2} x_{o(\Psi)-2},u^{o(\Psi)-1} x_{o(\Psi)-1}) .\] Then
\begin{align*}
A C_{\l I}(x_0,x_1,\dots,x_{o(\Psi)-2},x_{o(\Psi)-1})&=
 A(\l I(x_{o(\Psi)-1}),x_0,x_1,\dots,x_{o(\Psi)-2})\\
 &=(\l I(x_{o(\Psi)-1}),u  x_0,u^2 x_1,\dots,u^{o(\Psi)-1} x_{o(\Psi)-2})\\
 &=u(u^{o(\Psi)-1} I(x_{o(\Psi)-1}), x_0,u x_1,\dots,u^{o(\Psi)-2} x_{o(\Psi)-2})\\
 &= u C_IA(x_0,x_1,\dots,x_{o(\Psi)-2},x_{o(\Psi)-1}),\end{align*}
  i.e. \[A C_{\l I}= u C_IA.\]
 As~$A$ commutes with~$\Phi(w)$ for any~$w\in \mathrm{W}_F$, it defines an isomorphism between the Deligne representations 
~$(\Phi(\Psi),C_{\l I})$ and~$(\Phi(\Psi), uC_I)$, hence the result.
\end{proof}

\begin{notation}
It thus makes sense to write~$C(\Z_\Psi)=[C(\Z_\Psi,I)]$.
\end{notation}

We now want to show that all elements of~$\Irr_{\ss}(\WD, \Fl)$ with bijective Deligne operator are in such a class. Let~$(\Phi,U)$ be an irreducible~$\W_F$-semisimple Deligne representation with~$U$ bijective, and set~$V=V_\Phi$. By Lemma \ref{indec supported on a line}, we know that~$(\Phi,U)$ is supported on an irreducible line~$\Z_\Psi$.  Again 
we set~$W(i)=W(i,\Psi)$ for any~$\W_F$-stable subspace~$W$ of~$V$. We will now gather some information on the structure of 
$(\Phi,U)$. The first basic observation is that the relation \[\nu(w) U\Phi(w)= \Phi(w) U\] for~$w\in \W_F$, which can be rewritten 
\[(U \nu(w)\Phi(w)U^{-1}) U= \Phi(w) U,\] shows that~$U$ sends~$V(i)$ to~$V(i+1)$, in a necessarily bijective manner. We deduce the following. 

\begin{lemma}\label{basic structure of WD reps}
We have~$V(i)=U^i(V(0))$, for~$i=0,\dots,o(\Psi)-1$. In particular all the~$\W_F$-isotypic components~$V(i)$ have the same dimension for~$i=0,\ldots,o(\Psi)-1$. 
\end{lemma}

We now give a useful description the endomorphism ring of~$(\Phi,U)$.  In order to do so, we fix an isomorphism of~$\W_F$-modules \[J:V(0)\xrightarrow{\sim} V(0)\]  from~$\nu^{o(\Psi)} \Phi\mid _{V(0)}$ to~$\Phi\mid_{V(0)}$. Notice that as~$\Phi\mid_{V(0)}$ is the direct sum of~$m$ copies of~$\Psi$ (for some~$m\geq 1$), then \[A_0:=\End_{\mathrm{W}_F}(\Phi\mid_{V(0)})\] is isomorphic to 
$\M(m,R)$. This implies that all automorphisms of~$A_0$ are inner, in particular as~$A_0$ is also equal (not just isomorphic) to~$\End_{\mathrm{W}_F}(\nu^{o(\Psi)} \Phi\mid_{V(0)})$, there is~$P\in A_0^\times$ such that 
\[J^{-1} A J=P^{-1} A P,\] for all~$A\in A_0$. We fix such a~$P$ and put\[Y:=(U^{o(\Psi)}\mid_{V(0)})J P^{-1}\in A_0.\] 

\begin{prop}\label{algebra isomorphism}
The algebra~$C_{A_0}(Y)$ of elements of~$A_0$ commuting with~$Y$ is isomorphic to~$\End_{\WD}(V)$ via the map
\[L_0\mapsto L=L_0\oplus\dots \oplus L_{o({\Psi})-1}\] where \[L_i=U^i L_0 U^{-i}\in \End_{\mathrm{W}_F}(\Phi\mid_{V(i)}).\]
\end{prop}
\begin{proof}
Take~$L\in \End_{\WD}(V)$, then~$L$ stabilizes each~$V(i)$ as it commutes with the action of~$\W_F$, and we denote by 
$L_i$ the restriction~$L\mid_{V(i)}$.  Because~$L$ commutes with~$U$ we have \[L_i=U^i L_0 U^{-i}\in \End_{\mathrm{W}_F}(\Phi\mid_{V(i)}),\] and moreover~$L_0$ must commute with~$U_0:=(U^{o(\Psi)})\mid_{V(0)}$. Notice that~$U_0\in \End_{\Fl}(V(0))$ but has no reason to belong to~$A_0$. 

Conversely given~$L_0\in A_0$ commuting with~$U_0$, the map \[L=L_0\oplus\dots \oplus L_{o({\Psi})-1}\] with 
$L_i=U^i L_0 U^{-i}$ belongs to~$\End_{\WD}(V)$. Hence we have 
an isomorphism from the subalgebra 
\[B_0=\{L_0\in A_0,\ L_0U_0=U_0L_0\}\] of~$A_0$ to~$\End_{\WD}(V)$ defined by~$L_0\mapsto L$. Now the map 
$X=U_0 J$ belongs to~$A_0$, hence the relation \[L_0U_0=U_0L_0\] is equivalent to
\[L_0U_0=U_0 J J^{-1} L_0\Leftrightarrow L_0 X J^{-1}=X J^{-1} L_0\Leftrightarrow L_0 X =X J^{-1} L_0 J\Leftrightarrow 
 L_0 X =X P^{-1} L_0 P\]
\[\Leftrightarrow  L_0 X P^{-1}= X P^{-1} L_0 .\] This ends the proof.
\end{proof}

Let~$F$ be a field,~$X\in \M(n,F)$, and set 
\[C_{\M(n,F)}(X):=\{Y\in \M(n,F),\ YX=XY\},\] 
the subalgebra of matrices commuting with~$X$.  For use in the next proof, we recall the following lemma from basic linear algebra.

\begin{lemma}\label{commuting 1}
The algebra~$C_{\M(n,F)}(X)$ has dimension~$1$ if and only if~$n=1$.
\end{lemma}
\begin{proof}
Suppose that~$C_{\M(n,F)}(X)$ has dimension one, then in particular~$F[B]\simeq F[X]/(f_B)$ where~$f_B$ is the minimal polynomial of~$B$ must be one dimensional. Hence~$f_B$ has degree~$1$ and~$B$ is a multiple of~$I_n$. The statement follows.
\end{proof}

A corollary of Proposition \ref{algebra isomorphism} and Lemma \ref{commuting 1} is the following:

\begin{corollary}\label{crux}
We have~$\Phi(i)\simeq \nu^i\Psi$ for~$i=0,\dots,o(\Psi)-1$.
\end{corollary}

\begin{proof}
By Schur's lemma, the ring~$\End_{\WD}(V_\Phi)$ is one dimensional, hence by 
Proposition \ref{algebra isomorphism}, the algebra~$C_{A_0}(Y)$ is one-dimensional. However, as~$A_0\simeq \M(m,\Fl)$, if 
\[\Phi=\underbrace{\Psi\oplus\dots \oplus \Psi}_{m\times},\] Lemma \ref{commuting 1} implies that 
$m=1$ i.e.~$\Phi(0)\simeq \Psi$, and~$\Phi(i)\simeq \nu^i\Psi$ by Lemma \ref{basic structure of WD reps}.
\end{proof}

We obtain the main result of this section.

\begin{thm}\label{simple}
Take~$(\Phi,U)\in \Irr_{\ss}(\WD, \Fl)$ with~$U$ bijective, then there is a unique 
irreducible line~$\Z_\Psi$ of~$\W_F$, such that 
\[[\Phi,U]=C(\Z_\Psi).\]
\end{thm}
\begin{proof}
We already explained that~$(\Phi,U)$ is supported on an irreducible line~$\Z_\Psi$, which is necessarily unique by uniqueness 
of the decomposition of~$\Phi$ into a direct sum if elements in~$\Irr(\W_F,\Fl)$. Set~$V=V_\Phi$ and~$V(i)$ its~$\nu^i\Psi$-isotypic component. Take~$I\in \Iso_{\W_F}(\nu\Psi,\Psi)$ and consider~$C(\Psi,I)=(\Phi(\Psi),C_I)$. By Corollary \ref{crux},~$V(i)\simeq \nu^i\Psi$. Let~$J_0$ be an isomorphism between~$V(0)$ and~$\Psi$, then for each~$i$ between~$0$ and~$o(\Psi)-1$, the map~$J_i=C_I^iJ_0 U^{-i}$ is a linear bijection from~$W(i)$ to~$\nu^i\Psi$ commuting with the action of~$\W_F$. Hence the map~$J=\bigoplus_{i=0}^{o(\Psi)-1} J_i$ is a linear bijection from~$W$ to~$V$ commuting with the action of~$\W_F$.  It intertwines the action of 
$U$ and~$C_I$ on each~$W(i)$: for~$i=0,\dots, o(\Psi)-2$, we have \[JU\mid_{W(i)}=C_IJ\mid_{W(i)}=C_{\mu I}J\mid_{W(i)}\] 
for any~$\mu\in \Fl^\times$. On~$W(o(\Psi)-1)$, by Schur's lemma, there is~$\l\in \Fl^\times$ such that 
\[J U\mid_{W(o(\Psi)-1)}=\l C(I_0)J\mid_{W(o(\Psi)-1)}=C_{\l I}J\mid_{W(o(\Psi)-1)}.\] This shows that 
$J$ is an isomorphism between~$(\Phi,U)$ and~$C(\Psi,\l I)$. 
\end{proof}

\subsection{Indecomposable~$\W_F$-semisimple Deligne representations}

Take~$(\Phi,U)\in\Indec_{\W_F,\ss}(R)$, in particular it is supported on~$\Z_\Psi$ for~$\Psi$ an irreducible representation of~$\W_F$. As before we write~$V=V_\Phi$ and~$U=D+N$, and we already observed that~$U=N$ when~$R=\Fl$. We consider~$R=\Fl$ for a moment, the Deligne relation satisfied by~$U$ implies that the nonzero eigenvalues of~$D$ can be partitioned into orbits of~$q^\Z=q^{\Z/o(\nu)\Z}\simeq \Z/o(\nu)\Z$ acting by multiplication.

\begin{lemma}\label{indec nilp or bijective}
For~$(\Phi,U)\in\Indec_{\W_F,\ss}(\Fl)$, the eigenvalues of~$U$ lie in a single orbit under this action of~$\Z/o(\nu)\Z$, in particular~$U$ is either bijective or nilpotent.
\end{lemma}
\begin{proof}
We write~$U=D+N$ as before:
\[\Sp(D)=\{0\}\sqcup O_1 \sqcup \dots \sqcup O_s\] as the union of the orbit of~$0$ and the~$s$ orbits of nonzero eigenvalues. We set~$V_0=\Ker(D)$, and~$V_i=\bigoplus_{\l\in O_i} V_{\l}$ so that ~$V=\bigoplus_{i=0}^s V_i$ is a decomposition into~$\W_F$-stable summands.  The summands are stable under~$U$ as~$U$ commutes with~$D$. Hence the decomposition~$V=\bigoplus_{i=0}^s V_i$ is a direct sum of Deligne representations. As~$(\Phi,U)$ is indecomposable, we must have~$V=V_i$ for some~$i$.\end{proof}

We go back to general~$R$. Let us first consider the case~$U$ nilpotent, and start by a classical example.

\begin{definition}\label{segment}
Take~$m\geq 1$, and denote by~$[0,r-1]$ the~$\W_F$-semisimple Deligne representation~$(\Phi(r),N(r))$ where 
\[\Phi(r)=\bigoplus_{k=0}^{r-1} \nu^k,\] and 
\[N(r)(x_0,\ldots,x_{r-1})=(0,x_0,\ldots,x_{r-2}).\]
\end{definition}

\begin{notation}\label{general segment}
For~$a\leq b$ in~$\Z$, we set~$[a,b]=\nu^a [0,b-a]$.
\end{notation}

\begin{lemma}\label{segment indec}
$[0,r-1]\in \Indec_{\ss}(\WD, R)$.
\end{lemma}
\begin{proof}
In fact it is already indecomposable as a~$R[N(r)]$-module, indeed~$\End_{R[N(r)]}(V_{[0,r-1]})$ is equal to~$R[N(r)]$ is cyclic, 
and~$R[N(r)]=R[X]/(X^r)$ is a local ring.
\end{proof}

More generally, we have: 

\begin{prop}\label{Sp are indec}
Take~$\Psi\in \Irr(\W_F,R)=\Irr_{\WD,\ss}(R)\cap \Nilp_{\WD,ss}(R)$ and~$r\geq 1$, then 
\[[0,r-1]\otimes \Psi\in \Indec_{\ss}(\WD, R).\]
Moreover, take~$N(r)$ as in Definition \ref{segment}. If~$R=\Fl$, and~$I\in \Iso_{\W_F}(\nu^{o(\Psi)}\Psi,\Psi)$, then \[\End_{\WD}([0,r-1]\otimes \Psi)=\Fl[N(r)^{o(\Psi)}\otimes I^{-1}].\] 
If~$R=\Ql$, then \[\End_{\WD}([0,r-1]\otimes \Psi)\simeq \Ql.\] 
In both cases they are local rings.

\end{prop}
\begin{proof}
First, by Lemma \ref{good case tensor},~$[0,r-1]\otimes \Psi\in \Rep_{\ss}(\WD, R)$.
Let's compute the endomorphism algebra of~$[0,r-1]\otimes \Psi$. We do the case~$R=\Fl$, the case~$R=\Ql$ being similar. 
One has \[\End_{\Fl}([0,r-1]\otimes \Psi)=\End_{\Fl}([0,r-1])\otimes \End_{\Fl}(\Psi).\] The Deligne operator~$N(\Psi,r)$ of~$[0,r-1]\otimes \Psi\in \Indec_{\ss}(\WD, R)$ is~$N(r)\otimes \Id$, hence 
\[\End_{\Fl[N(\Psi,r)]}([0,r-1]\otimes \Psi)=\End_{\Fl[N(r)]}([0,r-1])\otimes \End_{\Fl}(\Psi)=\Fl[N(r)]\otimes \End_{\Fl}(\Psi).\]
Fixing~$I\in \Hom(\nu^{(o(\Psi)}\Psi,\Psi)$, and using the basis~$\Id, N(r),\ldots, N(r)^{r-1}$ of~$\Fl[N]$, one checks that 
the subalgebra of~$\Fl[N(r)]\otimes \End_{\Fl}(\Psi)$ commuting with the action of~$\W_F$ is 
$\F_l[N(r)^{o(\Psi)}\otimes I^{-1}]$, i.e. 
 \[\End_{\WD}([0,r-1]\otimes \Psi)=\Fl[N(r)^{o(\Psi)}\otimes I^{-1}].\] When~$R=\Ql$ we find 
 \[\End_{\WD}([0,r-1]\otimes \Psi)\simeq \Ql.\] In both cases, these algebras are local as they are of the form~$R[X]/(X^l)$ for~$l\geq 0$.
\end{proof}

Now we check that these are the only indecomposable~$\W_F$-semisimple Deligne representations with nilpotent Deligne operator. 
If~$N$ is a nilpotent endomorphism of an~$R$-vector space, we denote by~$\ind(N)$ its nilpotency index. 

\begin{thm}\label{V_0}
Let~$(\Phi,N)\in \Rep_{\ss}(\WD, \Ql)$, then there is a unique~$\Psi\in \Irr(\W_F,R)$ and a unique~$r\geq 1$ such that 
\[[\Phi,N]=(\Phi,N)=[0,r-1]\otimes \Psi=[0,r-1]\otimes [\Psi].\]
\end{thm}
\begin{proof}
The uniqueness of~$\Psi$,~$r$ and~$k$ are clear. Set~$r=\ind(N)$ and~$V=V_\Phi$. As in the proof of Proposition \ref{iso vs equiv nilp}, we construct~$\W_F$-stable 
subspaces~$S_i$ such that~$V=\bigoplus_{i=0}^{r-1} S_i$ and~$N(S_i)\subset S_{i-1}$.  Suppose 
 one~$S_{r-i}$ has length~$\geq 2$ as a~$\W_F$-module and take~$i=i_0$ minimal for this property. Write 
\[S_{r-i_0}=\Psi_{r-i_0}\oplus U_{r-i_0}\] with~$\Psi_{r-i_0}\in \Irr(\W_F,R)$ and~$U_{r-i_0}$ stable under~$\W_F$, 
then take a complement~$Z_{r-i_0-1}$ of \[N(\Psi_{r-i_0})\oplus N(U_{r-i_0})\] in~$S_{r-i_0-1}$, set 
\[U_{r-i_0-1}=N(U_{r-i_0})\oplus Z_{r-i_0-1}\] and continue. We construct 
a nontrivial~$(\Phi,N)$-stable decomposition 
\[[S_{r-1}\oplus \dots \oplus S_{r-(i_0-1)}\oplus \Psi_{r-i_0}\oplus \dots \oplus N^{r-i_0}(\Psi_{r-i_0})]
\oplus [U_{r-i_0}\oplus U_{r-i_0-1}\oplus \dots \oplus U_{0}]\] of~$V$, a contradiction. 
Hence each~$S_{r-i}$ is an irreducible representation of~$\W_F$. It then suffices to choose 
$\Psi=S_{r-1}$.
\end{proof}

It remains to consider the indecomposable~$\W_F$-semisimple Deligne representations~$(\Phi,U)$ with~$U$ invertible, in particular~$R=\Fl$. Let us start with an example.

\begin{prop}\label{construction indec}
Take an irreducible line~$\Z_\Psi$,~$I\in \Iso_{\W_F}(\nu^{o(\Psi)}\Psi,\Psi)$, and~$r\geq 1$, then 
\[[0,r-1]\otimes C(\Z_\Psi,I)\in \Indec_{\ss}(\WD,\Fl),\] and its Deligne operator~$U(r,I)=D(r,I)+N(r,I)$ with 
\[D(r,I)=\Id\otimes C_I\] and \[N(r,I)=N(r)\otimes \Id\] is bijective. Its endomorphism algebra is the local ring 
\[\End_{\WD}([0,r-1]\otimes C(\Z_\Psi,I))=\Fl[N(r)\otimes C_I^{-1}]=\Fl[N(r,I)\circ D(r,I)^{-1}].\]
\end{prop}
\begin{proof}
First notice that~$[0,r-1]\otimes C(\Z_\Psi,I)$ is indeed~$\W_F$-semisimple thanks to Lemma \ref{good case tensor}. Set 
$D=D(r,I)=\Id\otimes C_I$ and~$N=N(r,I)=N(r)\otimes \Id$. Then~$D$ and~$N$ clearly commute,~$N$ is nilpotent and 
$D$ is semisimple because~$C_I$ is ($C_I^{o(\nu)}$ is a nonzero scalar~$\mu$, and~$X^{o(\nu)}-\mu$ has simple roots 
because~$o(\nu)$ is prime to~$\ell$).~$U=D+N$ is bijective because~$D$ is, hence it remains to check that~$[0,r-1]\otimes C(\Z_\Psi,I)$ is indecomposable, we do this by looking at the endomorphism algebra again. We recall that the commutant of~$U$ is that of~$D$ intersected with that of~$N$. The commutant of~$N$ in 
$\End_{\Fl}([0,r-1]\otimes C(\Z_\Psi,I))$ is 
$\Fl[N]\otimes \End_{\Fl}(C(\Z_\Psi,I))$.  Hence we need to look at the commutant of the joint 
action of~$\W_F$ and~$D$ inside~$\Fl[N]\otimes \End_{\Fl}(C(\Z_\Psi,I))$. Writing an element in this commutant under the form 
$\sum_{k=0}^{r-1} N^k \otimes A_k$, the endomorphism~$A_k$ must belong to~$\Hom_{\W_F}(C(\Z_\Psi,I),\nu^k C(\Z_\Psi,I))$ and must commute with~$C_I$, i.e. \[C_I^{k}A_k \in \End_{\WD}(C(\Z_\Psi,I))\simeq \Fl,\] the latter isomorphism 
by Schur's lemma. Hence our element is of the form~$\sum_{k=0}^{r-1}\l_k N(r)^k \otimes C_I^{-k}$, and this implies that 
\[\End_{\WD}([0,r-1]\otimes C(\Z_\Psi,I))=\Fl[N(r)\otimes C_I^{-1}]=\Fl[N(r,I)\circ C(r,I)^{-1}].\]
As~$N(r)\otimes C_I^{-1}$ is nilpotent, this ring is local.
\end{proof}

We notice as in Lemma \ref{equiv irred}, the the equivalence class of such a Deligne representation 
is independent of~$I$. It is a consequence of Lemma \ref{good case tensor} and Lemma \ref{equiv irred}:

\begin{lemma}\label{equiv-indec}We have
\[[[0,r-1]\otimes C(\Z_\Psi,I)]=[0,r-1]\otimes [C(\Z_\Psi,I)]=[0,r-1]\otimes C(\Z_\Psi).\]
\end{lemma}

Now we move on to the description of~$[\Indec_{\ss}(D,\Fl)]$. We know by Lemma \ref{indec supported on a line} that 
$(\Phi,U)\in \Indec_{\ss}(D,\Fl)$ is always supported on an irreducible line. 
\textit{For the remainder of this section, we take~$(\Phi,U)\in \Indec_{\ss}(\WD,\Fl)$ supported on the irreducible line~$\Z_\Psi$ with~$U=D+N$ invertible}.

\begin{lemma}\label{semisimple part}
The Deligne representation~$(\Phi,D)\in \Rep_{\ss}(\WD, \Fl)$ is semisimple as a Deligne representation.
\end{lemma}
\begin{proof}
Let~$W$ be a~$(\Phi,D)$-stable subspace of~$V=V_\Phi$, and~$V_\l$ be the eigenspace of~$D$ associated to an eigenvalue~$\l$. As~$W$ is~$D$-stable, it is the direct sum of the eigenspaces \[W=\bigoplus_{\l\in \Sp(D)} W_\l\] where~$W_\l=W\cap V_\l$.   Moreover as~$W$ is~$\W_F$-stable, all the eigenspaces~$W_\l$ for~$\l$ in an orbit~\[O_\mu=\{q^\Z\mu\}\] have the same dimension, as~$\Phi(\Fr)$ sends~$W_\l$ to~$W_{q^{-1}\l}$.  Hence the direct sum \[W(\mu)=\bigoplus_{\l\in O_\mu} W_\l\] is~$\W_F$-stable and we write \[W=\bigoplus_{i=1}^r W(\mu_i).\]  For each~$\mu_i$, take a complement~$W'_{\mu_i}$ of~$W_{\mu_i}$ inside~$V_{\mu_i}$, so that 
$W'_{q^-k\mu_i}:=\Phi(\Fr)^k(W'_{\mu_i})$ is a complement of~$W_{q^{-k}\mu_i}$ inside~$V_{q^{-k}\mu_i}$. Then we 
set \[W'(\mu_i)=\bigoplus_{\l\in O_{\mu_i}} W'_{\l},\] and we set 
\[W'=\bigoplus_{i=1}^r W'(\mu_i).\] As each~$W'(\mu_i)$ is~$\W_F$-stable, so is~$W'$, and as~$W'$ is a direct sum of subspaces of the eigenspaces~$V_\l$, the space~$W'$ is also~$D$-stable. To conclude, 
notice that~$V=W\oplus W'$ by construction.
\end{proof}

Now we notice that~$(\Phi,D)$ is isotypic, i.e. it is the direct sum of isomorphic (not only equivalent) irreducible Weil-Deligne representations:

\begin{lemma}\label{isotypiclemma}
There is~$I\in \Iso_{\mathrm{W}_F}(\nu^{o(\Psi)}\Psi, \Psi)$ such that~$(\Phi,D)$ is isotypic of type 
$C(\Z_\Psi,I).$
\end{lemma}
\begin{proof}
Set~$N_D=ND^{-1}\in \End_{\mathrm{W}_F}(\Phi)$, and as~$N$ commutes with~$D$, in fact   
$N_D\in \End_{\WD}((\Phi,D))$. Now~$(\Phi,D)$ is the direct sum of its isotypic components of type~$C(\Z_\Psi,I_k)$ for 
\[I_k\in \Iso_{\mathrm{W}_F}(\nu^{o(\Psi)}\Psi,\Psi)\] thanks to Lemma \ref{semisimple part} and Theorem \ref{simple}. As~$N_D$ belongs to~$\End_{\WD}((\Phi,D))$, it stabilizes each of these and so does~$D$ by definition of the isotypic components, so they are in fact stable under~$U=D(Id+N_D)$. As~$(\Phi,U)$ is indecomposable there is only one of them.
\end{proof}

Using this lemma, we can obtain more information about the structure of indecomposable Deligne representations:

Put~$N_D=ND^{-1}$ as in the proof of Lemma \ref{isotypiclemma}, and take the model~$C(\Psi,I)=(\Phi(\Psi),C_I)$ for the type~$C(\Z_\Psi,I)$ of the isotypic Deligne representation~$(\Phi,D)$. Define \[H=\Hom_{\mathrm{D}}(C(\Psi,I),(\Phi,D)).\] 
We then define a Deligne representation structure~$(\Phi_0,U_0=D_0+N_0)$ on~$H \otimes_{\Fl} V_{\Phi(\Psi)}$ as follows:  
\[\Phi_0(w)=\Id\otimes \Phi(\Psi)(w)\] for~$w\in \mathrm{W}_F$, \[D_0=\Id \otimes C_I\] and 
\[N_0(h\otimes v)=N h C_I^{-1} \otimes C_I \cdot v = N_D h\otimes C_I\cdot v.\]
\begin{prop}\label{structure of indecomposable}
The Deligne representations~$(\Phi,U)$ and~$(\Phi_0,U_0)$ are isomorphic. In particular setting~$m=\ind(N)=\ind(N_D)$, we have~$m=d_\Phi/d_\Psi o(\Psi)$.
\end{prop}
\begin{proof}
We introduce the map~$\alpha: H \otimes_{\Fl} V_{C(\Psi,I)} \rightarrow V$ defined by 
$\alpha(\psi\otimes v)=\psi(v)$. One checks that~$\alpha$ intertwines the action of~$\W_F$, of 
$N$ and~$N_0$ and of~$D$ and~$D_0$. It is surjective by Lemma \ref{isotypiclemma}. By Schur's lemma,~$H$ has dimension 
$d_\Phi/d_\Psi o(\Psi)$, hence both spaces have the same dimension, so~$\alpha$ is bijective.  
\end{proof}

\begin{thm}\label{indec classif}
Take~$(\Phi,U)\in \Indec_{\ss}(\WD, \Fl)$ with~$U$ bijective, there are a unique irreducible line~$\Z_\Psi$ and a unique~$r\geq 1$ such that 
$[\Phi,U]=[0,r-1]\otimes C(\Z_\Psi)$.
\end{thm}
\begin{proof}
Set~$r=d_\Phi/d_\Psi o(\Psi)$. We define~$(\Phi_0,U_0)$ as in Proposition \ref{structure of indecomposable}, it is enough to show that 
\[[\Phi_0,U_0]=[0,r-1]\otimes C(\Z_\Psi).\] Write~$V_{[0,r-1]}=\mathrm{Vect}_{\Fl}(e_0,\dots,e_{r-1})$ with~$e_i=N(r)^i(e_0)$, and 
$\W_F$ acting on~$D_i=\mathrm{Vect}_{\Fl}(e_i)$ as~$\nu^i$. 
Similarly write~$H=\mathrm{Vect}_{\Fl}(h_0,\dots,h_{r-1})$ with~$h_i=N_D^i(h_0)$. Denote by~$B$ the~$\Fl$-linear isomorphism 
sending~$V_{[0,r-1]}\otimes V_{\Phi}$ to~$H\otimes V_{\Phi}$ defined by
\[B(\sum_{i=0}^{r-1} e_i\otimes v_i)= \sum_{i=1}^{r-1} h_i\otimes C_I^i v_i.\] Then 
\[B\in \Iso_{\WD}([0,r-1]\otimes C(\Psi,I),(\Phi_0,U_0)).\]
The uniqueness of~$r$ and~$\Z_\Psi$ are immediate.
\end{proof}

\subsection{Tensor product for~$\W_F$-semisimple Deligne representations}\label{tensor section}

We already saw the two problems of the tensor product in our setting: by example \ref{counter examples} (3) it does not preserve 
$\W_F$-semisimplicity, and by example \ref{non equiv tensor} even when it does it does not preserve equivalence classes. In this 
section we take care of those two problems, starting with the first one. These problems occur only when~$R=\Fl$ (see the beginning of Section \ref{tensor for W_F} and remark \ref{good case tensor}), so for this section~$R=\Fl$.

First we define the~$\W_F$-semi-simplification~$(\Phi,D)_{\ss}=(\Phi_{\ss},D_{\ss})$ of~$(\Phi,D)$ when~$D$ is semisimple.

\begin{prop}\label{C semi-simplification}
Let~$(\Phi,D)$ be a Deligne representation with~$D$ semisimple.
\begin{enumerate}[{(1)}]
\item  There is a filtration 
\[\{0\}=V_0\subset V_1\subset \dots \subset V_n=V_{\Phi}\]
of~$V_{\Phi}$ by~$(\Phi,D)$-stable subspaces such that the induced Deligne representation 
$(\Phi_i,D_i)$ on~$V_{i}/V_{i-1}$ is in~$\Irr_{\ss}(\WD, \Fl)$.
\item The 
$\W_F$-semisimple Deligne representation \[(\Phi,D)_{\ss}= (\Phi_{\ss},D_{\ss}):=\bigoplus_{i=1}^n (\Phi_i,D_i)\] is independent of the filtration, and~$C_{\ss}$ and~$C$ are semisimple with the same characteristic polynomial.\end{enumerate}
\end{prop}
\begin{proof}
By induction, it is enough to show that~$(\Phi,D)$ contains an irreducible~$\W_F$-semisimple Deligne subrepresentation. 
The unicity follows from standard facts on Jordan-H\"older composition series. Take~$\Psi$ an irreducible~$\W_F$-subrepresentation of~$\Phi$. Then~$\sum_{k\geq 0} C^k(\Psi)$ is a~$\W_F$-semisimple Deligne subrepresentation of~$(\Phi,D)$. It thus contains an indecomposable~$\W_F$-semisimple Deligne subrepresentation~$(\Phi_0,D\mid_{\Phi_0})$. But~$D\mid_{\Phi_0}$ being semisimple, there is~$\Psi$ an irreducible representation of~$\W_F$ such that~$(\Phi_0,D\mid_{\Phi_0})$ is either of the form~$\Psi=(\Psi,0)$ or of the form~$C(\Z_\Psi,I)$ for~$I$ an isomorphism between~$\nu^{o(\Psi)}\Psi$ and~$\Psi$, thanks to Proposition \ref{V_0} and Theorem \ref{indec classif}.
\end{proof}

We can now define the operation~$\otimes_{\ss}$ in~$\Rep_{\ss}(\WD, \F_l)$. 

\begin{definition}
Let~$(\Phi,U)$ and~$(\Phi',U')$ belong to~$\Indec_{\ss}(\,\Fl)$. Write~$(\Phi,U)$ as~
$[0,r_1-1]\otimes (\Phi_1,C_1)$ and~$(\Phi',U')$ as~$[0,r_2-1]\otimes (\Phi_2,C_2)$ with~$(\Phi_i,C_i)$ in 
$\Irr_{\ss}(\WD, \F_l)$ thanks to Proposition \ref{V_0} and Theorem \ref{indec classif}. In particular~$C_i$ is semisimple so~$C_1\otimes Id \oplus I_d\otimes C_2$ as well. Hence 
$[(\Phi_1,C_1)\otimes (\Phi_2,C_2)]_{\ss}$ is well defined according to Lemma \ref{C semi-simplification}. 
We then set 
\[(\Phi,U)\otimes_{\ss}(\Phi',U')= [0,r_1-1]\otimes [0,r_2-1] \otimes [(\Phi_1,C_1)\otimes (\Phi_2,C_2)]_{\ss}.\]
We the extend~$\otimes_{\ss}$ as an operation from~$\Rep_{\ss}(\WD, \Fl)\times \Rep_{\ss}(\WD, \Fl)$ to 
$\Rep_{\ss}(\WD, \Fl)$ by distributivity with respect to~$\otimes$.
\end{definition}

\begin{rem}
The operation~$\otimes_{\ss}$ is bilinear by definition, and is associative and commutative, these properties inherited from~$\otimes$.
\end{rem}

We are now ready to define~$[\Phi,U]\otimes [\Phi',U']$ for any pair of Deligne representations in~$\Rep_{\ss}(\WD,\Fl)$. 
We first consider the indecomposable case. 
We recall that if~$U=D+N$, then~$D$ and~$N$ are polynomials in~$U$, and that the projections onto an eigenspace of~$D$ with respect to the sum of its other eigenspaces is a polynomial in~$D$, hence in~$U$.

\begin{prop}\label{tensor of equivalence classes}
Let~$(\Phi,U)$ and~$(\Phi',U')$ be two~$\W_F$-semisimple indecomposable Deligne representations with~$U$ and~$U'$ bijective, then for all~$(\l,\mu)$ in a subset~$Z(U,U')$ of~$(\Fl^\times)^2$ consisting of all elements outside a finite number of hyperplanes of 
~$\Fl^2$, all representations \[(\Phi,\l U)\otimes_{\ss} (\Phi',\mu U')\] are equivalent to one another (say equivalent to~$(\Phi,\l_0 U)\otimes_{\ss} (\Phi',\mu_0 U')$), and we set
\[[\Phi,U]\otimes_{\ss} [\Phi',U']=[(\Phi,\l_0 U)\otimes_{\ss} (\Phi',\mu_0 U')].\]
\end{prop}
\begin{proof}
We write~$(\Phi,U)=[0,r_1-1]\otimes (\Phi_1,D_1)$ and 
$(\Phi',U')=[0,r_2-1]\otimes (\Phi_2,D_2)$ indecomposable, with~$(\Phi_i,D_i)$ irreducible. 
Let's first do the irreducible case, we are going to show that
\[(\Phi_1,\l D_1)\otimes_{\ss} (\Phi_2,\mu D_2)=((\Phi_1,\l D_2)\otimes (\Phi_2,\mu D_2))_{\ss}\] is equivalent to 
\[(\Phi_1, D_1)\otimes_{\ss} (\Phi_2, D_2)=((\Phi_1, D_1)\otimes (\Phi_2, D_2))_{\ss}\] for~$(\l,\mu)\in(\Fl^\times)^2$ outside a finite number of hyperplanes of~$(\Fl)^2$. We set \[\mathrm{Spec}(C_1)-\{0\}=\{a_1,\dots,a_m\},\] \[\mathrm{Spec}(C_2)-\{0\}=\{b_1,\dots,b_n\},\] and \[Z(D_1,D_2)=\{(\l,\mu) \in (\Fl^\times)^2, \ \forall \ i,\ j ,\ k, \ l, \ \l a_i+ \mu b_j \neq 0, \ \l a_i+ \mu b_j \neq \l a_k+ \mu b_l\}.\] The 
For all~$(\l,\mu)$ in~$Z(D_1,D_2)$, all 
$D(\l,\mu)=\l D_1\otimes \Id+ \Id\otimes \mu D_2$ have the same eigenspaces in the sense that they have the same kernel, and for~$(\l,\mu)$ and~$(\l',\mu')$ in~$Z(D_1,D_2)$, the eigenspace of~$D(\l,\mu)$ corresponding to~$\l a_i+\mu b_j$ is the 
same space as the eigenspace of~$D(\l',\mu')$ corresponding to~$\l' a_i+\mu' b_j$. In particuliar~$D(\l,\mu)$ and~$D(\l',\mu')$ are polynomials in one another. This implies that for~$(\l,\mu)$ and~$(\l',\mu')$ in~$Z(D_1,D_2)$, if~$(\Phi_*,U_*)$ is an irreducible subquotient of 
$((\Phi_1 \otimes \Phi_2), D(\l,\mu))$ (see Proposition \ref{C semi-simplification}) with~$U_*$ the Deligne endomorphism of~$V_{\Phi_*}$ induced by~$D(\l,\mu)$, then~$(\Phi_*,T_*)$ is also an irreducible subquotient of 
$((\Phi_1 \otimes \Phi_2), D(\l',\mu'))$ with~$T_*$ the Deligne endomorphism of~$V_{\Phi_*}$ induced by~$D(\l',\mu')$. As 
$U_*$ and~$T_*$ are bijective together, or zero together,  
Theorem \ref{simple} and Lemma \ref{irred nilp} imply that~$[\Phi_*,U_*]=[\Phi_*,T_*]$ (both representations 
have the same dimension and are supported on the same irreducible line). Now let's go back to the indecomposable case.

For~$(\l,\mu)$ and~$(\l',\mu')$ in~$Z(D_1,D_2)$: 
\[(\Phi,\l U)\otimes_{\ss} (\Phi',\mu U')=
(\Phi(r_1),\l N(r_1))\otimes (\Phi(r_2),\mu N(m_2)) \otimes (\Phi_1,\l D_1)\otimes_{\ss} (\Phi_2,\mu D_2)\]
\[=(\Phi(r_1),\l N(r_1))\otimes (\Phi(r_2),\mu N(m_2)) \otimes (\Phi_1,\l D_1)\otimes_{\ss} (\Phi_2,\mu D_2)\]
according to Proposition \ref{iso vs equiv nilp}. But 
\[(\Phi_1, \l D_1)\otimes_{\ss} (\Phi_2,\mu D_2)\sim (\Phi_1,\l' D_1)\otimes_{\ss} (\Phi_2,\mu' D_2)\] and 
the~$\W_F$-representation~$\Phi(r_1)\otimes \Phi(r_2)$ is a direct sum of characters, hence this implies 
\[(\Phi,\l U)\otimes_{\ss} (\Phi',\mu U')\sim (\Phi(r_1),N(r_1))\otimes (\Phi(r_2), N(r_2)) \otimes (\Phi_1,\l' D_2)\otimes_{\ss} (\Phi_2,\mu' C_2),\] which is itself isomorphic to 
\[(\Phi(m_1),\l'N(m_1))\otimes (\Phi(m_2), \mu' N(m_2)) \otimes (\Phi_1,\l' C_2)\otimes_{\ss} (\Phi_2,\mu' D_2)= 
(\Phi,\l U)\otimes_{\ss} (\Phi',\mu U')\] thanks to Proposition \ref{iso vs equiv nilp}. Finally one can set 
$Z(U,U')=Z(D_1,D_2)$
\end{proof}

For the general case, as~$[\Phi,U]\oplus [\Phi',U']$ is a well defined operation from 
\[[\Rep_{\ss}(\WD, \Fl)] \times [\Rep_{\ss}(\WD, \Fl)] \] to~$[\Rep_{\ss}(\WD, \Fl)]~$, we then extend the definition of 
$\otimes_{\ss}$ by distributivity (bilinearity) from indecomposable to all elements of~$[\Rep_{\ss}(\WD, \Fl)]$.

\begin{rem}\label{tensor product preserves non banal}
If~$(\Phi,U)$ and~$(\Phi',U')$ are in~$\Rep_{\ss}(\WD,\F_l)$ with~$U$ or~$U'$ both bijective, and if~$(\Phi_0,U_0)\in [\Phi,U]\otimes_{\ss} [\Phi',U']$, then~$U_0$ is bijective. Indeed it suffices to check this when~$(\Phi,U)$ and~$(\Phi',U')$ are both indecomposable. It then follows from the proof of Theorem \ref{tensor of equivalence classes} that 
$D(\l,\mu)_{\ss}$ is bijective for~$(\l,\mu)\in Z(U,U')$, hence~$U_0$ too.
\end{rem}

\section{Local constants of Deligne representations}\label{Sectionlocalfactors}

\subsection{Definition and basic properties}

Now~$R$ is either~$\overline{\Q_\ell}$ or~$\overline{\Z_\ell}$. 

\begin{definition}
Let~$(\Phi,U)$ be a~$\W_F$-semisimple Deligne~$R$-representation, we set
\[L(X,(\Phi,U))=\det((\Id-X\Phi(\Fr))\vert_{\Ker(U)^{I_F}})^{-1}.\]
\end{definition}

\begin{rem}
In particular $L(X,(\Phi,U))=1$ whenever $U$ is bijective.
\end{rem}

In what follows~$\psi:F\rightarrow R^\times$ is a nontrivial character, and we use a~$\psi$-self-dual additive measure on~$F$.

\begin{definition}
Let~$\Psi$ be an irreducible representation of~$\W_F$. The epsilon factor~$\e(X,\Psi,\psi)$ is defined in \cite{Deligne73}. If~$\Psi$ is not an unramified character, we set~$\gamma(X,\Psi,\psi)=\e(X,\Psi,\psi)$. If~$\Psi=\chi$ is an unramified character, viewing~$\chi$ as a character of 
~$F^\times$, we take~$\gamma(X,\chi,\psi)$ to be the Tate gamma factor defined in Tate's thesis when~$R=\Ql$ and in 
 \cite{MinguezZeta} or \cite{KM17} when~$R=\Fl$.
\end{definition}

\begin{definition}
Let~$(\Phi,U)$ be a~$\W_F$-semisimple Deligne representation, with~$\Phi=\bigoplus_{i=1}^r \Psi_i$ the underlying 
direct sum of irreducible representations of~$\W_F$, we set 
\[\gamma(X,(\Phi,U),\psi)=\prod_{i=1}^r \gamma(X,\Psi_i,\psi).\]
\end{definition}

\begin{definition}
Let~$(\Phi,U)$ be a~$\W_F$-semisimple Deligne representation, we set 
\[\e(X,(\Phi,U),\psi)=\gamma(X,(\Phi,U),\psi)\frac{L(X,(\Phi,U))}{L(q^{-1}X^{-1},(\Phi,U)^{\vee})}.\]
\end{definition}

The following proposition follows from the definitions.

\begin{prop}\label{multiplicative}
The maps~$(\Phi,U)\rightarrow L(X,(\Phi,U))$,~$(\Phi,U)\rightarrow \gamma(X,(\Phi,U),\psi)$, hence the map 
$(\Phi,U)\rightarrow \e(X,(\Phi,U),\psi)$ are multiplicative with respect to direct sums. 
\end{prop}

\begin{corollary}\label{local constants depend on equiv only}
Two equivalent Deligne representations have the same constants~$L$,~$\gamma$ and~$\e$, in particular we can talk of the local 
constants~$L$,~$\gamma$ and~$\e$ of~$[\Phi,U]\in [\Rep_{\ss}(\WD, \F_l)]$.
\end{corollary}
\begin{proof}
By Proposition \ref{multiplicative} one reduces to the indecomposable case. But if~$(\Phi,U)$ is indecomposable, then~$L(X,(\Phi,U))$ and~$L(X,(\Phi,\l U))$ are equal for~$\l\neq 0$ as~$\Ker(U)=\Ker(\l U)$, and the definition of~$\gamma(X,(\Phi,U),\psi)$ and~$\gamma(X,(\Phi,\l U),\psi)$ does not see~$U$ and~$\l U$. Finally, the result for~$\e$ follows.
\end{proof}

The next proposition also follows from the definitions and the known multiplicativity properties of 
$\gamma$-factors in the~$\ell$-adic case.

\begin{prop}\label{recover usual constants in char zero}
When~$R=\Ql$, the factors defined here are the usual~$L$,~$\gamma$ and~$\e$ factors from \cite[8.12]{Deligne73}, see \cite{HM17} for~$\gamma$. 
When~$R=\Fl$, if~$\Phi$ is irreducible as a representation of~$\W_F$ and~$U=0$, then \[L(X,(\Phi,0))=L(X,\Phi)\] is the usual Artin~$L$-factor of~$\Phi$, and \[\e(X,(\Phi,0),\psi)= \e(X,\Phi ,\psi)\] is also the usual~$\e$-factor defined in \cite{Deligne73}.
\end{prop}
\begin{proof} 
There second assertion follows from the definitions. Suppose that~$R=\Ql$, the quickest way to check our the first assertion is to use \cite[Lemma 4.4]{HM17}. Let~$\gamma'$ be the~$\gamma$ factor defined in [ibid. (4.6)], and~$(\Phi,U)$ a~$\W_F$-semisimple Deligne representation 
of~$\W_F$. By [ibid. Lemma 4.4], one has 
\[\gamma'(X,(\Phi,U),\psi)=\gamma'(X,\Phi,\psi):=\e(X,\Phi,\psi)\frac{L(q^{-1}X^{-1},\Phi^\vee)}{L(X,\Phi)},\] 
where~$\e$ and~$L$ are the local constants attached to representations of~$\W_F$ in \cite{Deligne73}. In particular as they are multiplicative with respect to direct sum, this implies that~$\gamma'$ is, hence if~$\Phi=\bigoplus_{i=1}^r \Psi_i$ with 
$\Psi_i$ irreducible representations of~$\W_F$, one has \[\gamma'(X,(\Phi,U),\psi)=\prod_{i=1}^r \gamma'(X,\Psi_i,\psi).\] So to prove that~$\gamma$ is equal to~$\gamma'$, it suffices to check that~$\gamma'(X,\Psi,\psi)=\gamma(X,\Psi,\psi)$ when 
$\Psi$ is an irreducible representation of~$\W_F$. If~$\Psi$ is not an unramified character then 
\[\gamma'(X,\Psi,\psi)=\e(X,\Psi,\psi)=\gamma(X,\Psi,\psi).\] If~$\chi$ is an unramified character, the assertion for~$L$ follows from the definitions, hence the assertion for~$\e$ is equivalent to~$\gamma'=\gamma$, which is by definition of $\e$ in this case.
\end{proof}

When~$R=\Ql$, the constant~$\e(X,(\Phi,U),\Psi)$ is invertible in the ring~$R[X^{\pm 1}]$. Let's check this when 
$R=\Fl$. We will write~$P(X)\underset{\Fl[X^{\pm 1}]^\times}{\sim}Q(X)$ if~$P,\ Q \ \in \Fl(X)$ are equal up to 
an element in~$\Fl[X^{\pm 1}]^\times$.

\begin{prop}\label{epsilon is invertible}
Let~$(\Phi,U)$ be an~$\ell$-modular~$\W_F$-semisimple Deligne representation, then 
\[\e(X,(\Phi,U),\Psi) \in \Fl[X^{\pm 1}]^\times.\]
\end{prop}
\begin{proof}
Thanks to Proposition \ref{multiplicative}, and Corollary \ref{local constants depend on equiv only}, we can suppose that 
$(\Phi,U)=[0,r-1]\otimes \Psi$ or~$(\Phi,U)=[0,r-1]\otimes C(\Psi,I)$. Let's start with the second case.

As~$L(X,(\Phi,U))=1$ we have 

\[\e(X,(\Phi,U),\psi)=\gamma(X,(\Phi,U),\psi)=\prod_{k=0}^{o(\Psi)-1}\gamma(X,\nu^k\Psi,\psi)^r.\] 

If~$\Psi$ is either ramified, or a non banal character (i.e.~$o(\chi)$ which is equal to~$o(\nu)$ is~$1$, or equivalently~$q \equiv 1[\ell]$) then~$\gamma(X,\nu^k\chi,\psi)=\e(X,\nu^k\chi,\psi)$ 
and we are done. 

If~$\Psi$ is a banal unramified character~$\chi$ (i.e.~$o:=o(\nu)=o(\chi)\geq 2$) with~$\chi(\w)=t\in \Fl^\times$  
\[\prod_{k=0}^{o-1}L(X,\nu^k\chi)=\prod_{k=0}^{o-1}(1-tq^{-k}X)^{-1}=(1-(tX)^{o})^{-1}\] and similarly 
\[\prod_{k=0}^{o-1}L(q^{-1}X^{-1},\nu^{-k}\chi^{-1})=(1-(tX)^{-o})^{-1},\] hence 
\[\e(X,(\Phi,U),\psi)=\prod_{k=0}^{o-1}\e(X,\nu^k\chi,\psi)^r(-(tX)^o)^r\in \Fl[X^{\pm 1}]^\times.\]
It remains to deal with the case~$(\Phi,U)=[0,r-1]\otimes \Psi$. If~$\Psi$ is not an unramified character or if~$\Psi$ is a non banal 
unramified character then 
\[\e(X,(\Phi,U),\psi)=\gamma(X,(\Phi,U),\psi)=\prod_{i=0}^{r-1}\gamma(X,\nu^i\Phi,\psi)=
\prod_{i=0}^{r-1}\e(X,\nu^i\Phi,\psi)\in \Fl[X^{\pm 1}]^\times.\]
If~$\Psi$ is a banal unramified character~$\chi$ ($q\not \equiv 1[\ell]$) with~$\chi(\w)=t\in \Fl^\times$, then 
\[\e(X,(\Phi,U),\psi)=\gamma(X,(\Phi,U),\psi)\frac{L(X,(\Phi,U))}{L(q^{-1}X^{-1},(\Phi,U)^\vee)}=\gamma(X,(\Phi,U),\psi)
\frac{(1-t^{-1}q^{-1}X^{-1})}{(1-tq^{1-r}X)} \]

However \[\gamma(X,(\Phi,U),\psi)=\prod_{i=0}^{r-1}\gamma(X,\nu^i\chi,\psi)=\prod_{i=0}^{r-1}\e(X,\nu^i\chi,\psi)
\prod_{i=0}^{r-1}\frac{(1-t q^{-i}X )}{(1-t^{-1}q^{i-1}X^{-1})}\]
Now \[\prod_{i=0}^{r-1}\frac{(1-t q^{-i}X )}{(1-t^{-1}q^{i-1}X^{-1})}=
\frac{\prod_{i=0}^{r-1}(1-t q^{-i}X )}{\prod_{i=-1}^{r-2}(1-t^{-1}q^{i}X^{-1})}\underset{\Fl[X^{\pm 1}]^\times}{\sim}
\frac{(1-tq^{1-r}X)}{(1-t^{-1}q^{-1}X^{-1})},\] hence 
\[\e(X,(\Phi,U),\psi)\underset{\Fl[X^{\pm 1}]^\times}{\sim} 1.\]
\end{proof}

We finally recall the multiplicativity relation for the~$L$-factors of nilpotent Deligne representations.

\begin{lemma}\label{multiplicativity}
Let~$\Psi$ and~$\Psi'$ be two irreducible representations of~$\W_F$ (over~$R$), and~$n\geq m\geq 1$, then: 
\[L(X,([0,n-1]\otimes \Psi)\otimes_{\ss} ([0,m-1]\otimes \Psi'))=\prod_{k=0}^{m-1} L(X,\nu^{n-1}\Psi\otimes_{\ss} \nu^k\Psi') \]
\end{lemma}
\begin{proof}
By definition  
\[([0,n-1]\otimes \Psi)\otimes_{\ss} ([0,m-1]\otimes \Psi')=[0,n-1]\otimes [0,m-1]\otimes (\Psi \otimes_{\ss} \Psi').\]
Moreover, writing as usual~$[0,r-1]=(\Phi(r),N(r))$, as~$\Phi(n)^{I_F}=\Phi(n)$ and~$\Phi(m)^{I_F}=\Phi(m)$, we deduce that 
\[(\Phi(n)\otimes \Phi(m)\otimes (\Psi \otimes_{\ss} \Psi'))^{I_F}=\Phi(n)\otimes \Phi(m)\otimes (\Psi \otimes_{\ss} \Psi')^{I_F}.\]
Hence we can suppose that~$\Psi$ and~$\Psi'$ are both the trivial character to prove the multiplicativity relation. 
Writing~$\Phi(l)=\bigoplus_{i=0}^{l-1}\nu^i.e_i$, one can decompose 
$\Phi(n)\otimes \Phi(m)$ as~$\bigoplus_{j=0}^{n+m-2} V_k$, where~$V_k=\Vect(e_i\otimes e_j)_{{i+j}=k}$. Write 
$N=N(n)\oplus \Id+ \Id\otimes N(m)$, then~$N(V_k)\subset V_{k+1}$ for~$k<n+m-2$ and~$N(V_{n+m-2})=\{0\}$. Hence 
\[\Ker(N)=\bigoplus_{j=0}^{n+m-2} \Ker(N)\cap V_k.\] However one checks that \[\Ker(N)\cap V_k=\{0\}\] for~$j<n-1$, and 
 \[\Ker(N)\cap V_{n-1+k}=\Vect(\sum_{i=0}^{m-1} (-1)^i e_{n+m-2-i}\otimes e_{i} )\] for~$k=0,\dots, m-1$. Denoting~$K_k$ the space 
~$\Ker(N)\cap V_k$, we thus showed that 
\[\Ker(N)=\bigoplus_{k=0}^{m-1} K_{n-1+k}\] where~$\W_F$ acts via~$\nu^{n-1+k}$ on the line~$K_{n-1+k}$, and the result follows.
\end{proof}

\subsection{Reduction modulo~$\ell$ of Deligne constants}

First state two results concerning good reduction of local constants. 
The first is a result of \cite{HM17}, which states that the~$\gamma$ are compatible with reduction modulo~$\ell$. For 
$P\in \Zl[X]$, we denote by~$r_\ell(P)=\Fl[X]$ the polynomial obtained by applying~$r_\ell$ to the coefficients of~$P$. For~$Q$ in~$\Zl[X]$, such that~$r_\ell(Q)\neq 0$, we set~$r_\ell(P/Q)=r_\ell(P)/r_\ell(Q)$.

\begin{prop}\label{HM}\cite[Theorem 1.1, 1]{HM17}: Let~$\Phi$ be an integral semisimple representation of~$\W_F$, then 
\[r_\ell(\gamma(X,\Phi,\psi))=\gamma(X,r_\ell(\Phi),\psi).\]
\end{prop}

Then we state the following immediate consequence of Proposition \ref{banal tensor unramified reduction}.

\begin{thm}\label{banal Deligne L functions}
Let~$\Psi$ and~$\Psi'$ be banal irreducible representations of~$\W_F$, and let~$\tilde{\Psi}$ and~$\tilde{\Psi}'$ be irreducible 
$\ell$-adic lifts of such representations, then one has:

\[L(X,\Psi\otimes_{\ss}\Psi')=r_\ell(L(X,\tilde{\Psi}\otimes \tilde{\Psi}')).\]
\end{thm}
\begin{proof}
Because~$\tilde{\Psi}\otimes \tilde{\Psi}'$ is integral: 
\[\det((\Id-X\Phi(\Fr))\mid_{(V_\Psi\otimes V_{\Psi'})^{I_F}})\in \Zl[X],\] hence it makes 
sense to consider~$r_\ell(L(X,\tilde{\Psi}\otimes \tilde{\Psi}')).$

Moreover thanks to Proposition \ref{banal tensor unramified reduction}, we know that~$\overline{(V_{\tilde{\Psi}}\otimes V_{\tilde{\Psi}'})^{I_F}}=(V_\Psi\otimes V_{\Psi'})^{I_F}$, hence 
\[r_\ell(L(X,\tilde{\Psi}\otimes \tilde{\Psi}'))=\det((\Id-X\Phi(\Fr))\mid_{(V_\Psi\otimes V_{\Psi'})^{I_F}})^{-1}.\]
Finally according to [ibid.] again, we have 
$((V_\Psi\otimes V_{\Psi'})^{I_F})_{\ss}\simeq (V_\Psi\otimes_{\ss} V_{\Psi'})^{I_F}$, hence 
\[\det((\Id-X\Phi(\Fr))\mid_{(V_\Psi\otimes V_{\Psi'})^{I_F}})^{-1}=L(X,\Psi\otimes_{\ss}\Psi').\]
\end{proof}

\section{The modular Langlands correspondence and local constants}

In \cite{Viginv}, Vignéras defined a bijection $\V$ (Theorem \ref{Vcorrespondence}) between $\Irr_{\gen}(G,\Fl)$ and $\Nilp_{\ss}(\WD,\Fl)$, the ``semisimplification'' of which is obtained by ``reducing modulo $\ell$'' the $\ell$-adic $\LLC$, and which moreover commutes with character twists, taking duals, and 
takes central characters to determinant. The aim of this last section is to define an injection $\C$ (Definition \ref{Vcorrespondence}) of $\Irr_{\gen}(G,\Fl)$ into 
$[\Rep_{\ss}(\D,\Fl)]$, which besides sharing all these properties with $\V$, takes the local factors of pairs of generic representations defined in \cite{KM17} to those of tensor products of elements in $[\Rep_{\WD,\ss}]$ (Theorem \ref{preservation of local factors}).

\subsection{Representations of $\GL(n,F)$}

We put~$G_n=\GL(n,F)$ ($G_0$ is trivial by convention), and denote by~$N_n$ the group of unipotent upper triangular matrices in~$G_n$.  By abuse of notation, for~$n\in N_n$ and~$\psi:F\rightarrow R^\times$ a non-trivial character, we set
\[\psi(n)=\psi{\textstyle\left(\sum_{i=1}^{n-1} n_{i,i+1}\right)}.\]
We say that an irreducible representation~$\pi$ of~$G_n$ is \textit{generic} if 
\[\Hom_{N_n}(\pi,\psi)\neq \{0\},\] 
in which case one knows it is well-known that~$\dim_R(\Hom_{N_n}(\pi,\psi))=1$, but we do not use this multiplicity one fact. 

For $\{n_1,\ldots, n_r\}$ positive integers, let~$\pi_i\in \Rep(G_{n_i},R)$ and put~$n=\sum_{i=1}^r n_i$.  We denote by 
\[\pi_1\times \dots \times \pi_r\in \Rep(G_n,R)\] the normalized parabolic induction of the~$\pi_i$'s. 
An irreducible~$R$-representation of~$G_n$ is called \emph{cuspidal} it does not
appear as a subrepresentation of a properly parabolically induced representation. It is called
\emph{supercuspidal} if moreover it does not appear as a subquotient of such a representation. By classical results (see 
for example \cite{BZ77}, \cite{Z} and \cite{Vselecta}), if the~$\pi_i$'s are generic, then~$\pi_1\times \dots \times \pi_r$ has a unique generic subquotient and cuspidal representations are always generic. If~$\rho$ is cuspidal, we will denote by 
\[\St(r,\rho)\] the unique generic subquotient of \[\rho\times \nu\rho \times \dots \times \nu^{r-1}\rho.\] 
By convention $\St(0,\rho)$ is the trivial representation of the trivial group $G_0$.

On the other hand, in \cite{BZ77} when~$R=\Ql$, and \cite{Vselecta} or \cite[Definitions 7.5]{MSDuke} when~$R=\Fl$, to a 
\textit{cuspidal segment} 
\[[a,b]_\rho=(\nu^a\rho,\dots,\nu^b\rho)\] with~$a\leq b$, the authors attach a certain irreducible quotient~$\La([a,b]_\rho)$ and a certain irreducible submodule~$\Ze([a,b]_\rho)$ of \[\nu^a\rho \times \dots \times \nu^{b}\rho,\] which are respectively the cosocle and socle of it when~$R=\Ql$.

Consider two cuspidal segments~$\D$ and~$\D'$, we say that~$\D$ precedes~$\D'$ if one can extract a segment longer than both from the sequence~$(\D,\D')$, in which case we set~$\D\prec \D'$. We say that two cuspidal segments are \textit{linked} if one of them precedes the other one, otherwise we say that they are \textit{unlinked}.


We denote by~${}^*$ the Aubert-Zelevinsky involution on~$\Irr(G,R)$ (see \cite{Z,Aubertinvolution} when~$R=\Ql$ and \cite{Vinvolution,MSinvolution} when~$R=\Fl$), it satisfies 
\[(\pi_1\times \pi_2)^*=\pi_1^*\times \pi_2^*\] when~$\pi_1\times \pi_2$ is irreducible. This~$^*$ involution also commutes with taking duals: in the modular case, using the notations of \cite[Theorem 8]{MSinvolution}, this property follows from the theorem itself and the fact that~${\rm \mathbf{D}}$ commutes with taking duals. It is shown in these references 
(for example \cite[Proposition 4.10]{MSinvolution}) that for a cuspidal segment~$\D$:
\[\La(\D)=\Ze(\D)^*.\]

\begin{notation}
Let now us fix some more notation:
\begin{itemize}
\item~$\Irr(G,R)=\coprod_{n\geq 0}  \Irr(G_n,R)$.
\item~$\Irr_{\gen}(G_n,R)$: the generic classes in~$\Irr(G_n,R)$.
\item~$\Irr_{\gen}(G,R)=\coprod_{n\geq 0} \Irr_{\gen}(G_n,R)$.
\item~$\Irr_{\c}(G_n,R)$: the cuspidal classes in~$G_n$.
\item~$\Irr_{\sc}(G_n,R)$: the supercuspidal classes in~$G_n$.
\item~$\Irr_{\c}(G,R)=\coprod_{n\geq 0} \Irr_{\c}(\GL(n,F))$.
\item~$\Irr_{\sc}(G,R)=\coprod_{n\geq 0} \Irr_{\sc}(G_n,R)$. 
\item We denote by a right index~$e$ when we restrict to integral representations: for example~$\Irr(G,\Ql)_e$,~$\Irr_c(G,\Ql)_e$, etc.\end{itemize}
\end{notation}

We denote by $c_\pi$ the central character of $\pi\in \Irr(G,R)$. By \cite[II.4.12]{Vbook},~$\Irr_c(G,\Ql)_e$ are the elements in~$\Irr_c(G,\Ql)$ with integral central character.

If~$\rho$ is a cuspidal representation, we denote by~$\Z_\rho$ the associated \emph{cuspidal line} 
\[\Z_\rho=\{\nu^k\rho,\ k\in\Z\}.\] If $\rho$ is supercuspidal, we say that $\pi\in \Irr(G,R)$ is \textit{supported on $\Z_\rho$} if all supercuspidal representations of its supercuspidal support (which exists by \cite{Vselecta} or \cite{MSDuke}) belong to $\Z_\rho$. The set $\Z_\rho$ is finite if and only if~$R=\Fl$, in which case we set~$o(\rho)=|\Z_\rho|$. Following \cite[Remarque 8.15]{MSDuke}, we say that $\pi\in \Irr(G,R)$ is banal if the cuspidal support of 
$\pi$ contains no cuspidal line, in particular non banal irreducible representations exist only when $R=\Fl$ and a cuspidal representation $\rho$ is non banal if and only if $o(\rho)=1$. By \cite{Vbook}, \cite{Vselecta} or \cite[Theorem 6.4]{MSDuke}, if a cuspidal~$\rho$ is banal, then it is 
supercuspidal. If~$\tau$ is cuspidal non supercuspidal (which happens only when~$R=\Fl$), then there is a non-negative integer~$k$ such that
\[\tau=\St(o(\rho)\ell^k,\rho)\] for a supercuspidal representation~$\rho$, the cuspidal line of which 
is unique ($\rho$ can be replaced by any supercuspidal representation on the same line). Therefore, in this case, we set
\[\St_k(\Z_\rho)=\St(o(\rho)\ell^k,\rho).\]

If~$\rho$ is cuspidal, then~$\La([0,r-1]_\rho)=\St(r,\rho)$ if and only if either~$r<o(\rho)$ when~$\rho$ is banal, or 
$r<\ell$ when~$\rho$ is non-banal (\cite[Remarque 8.14]{MSDuke}.

By \cite[Theorem 9.7]{Z} when~$R=\Ql$, and \cite[Theorem V.7]{Vselecta} or \cite[Theorem 9.10]{MSDuke} when 
$R=\Fl$, a representation~$\pi\in \Irr_{\gen}(G)$ can be written under the form of a commutative product 
\[\pi=\St(m_1,\rho_1)\times \dots \times \St(m_r,\rho_r)\] where, for~$i\in\{1,\ldots,r\}$,  the cuspidal segments 
$[0,m_i-1]_{\rho_i}$ are unlinked and unique up to ordering.

In particular, for~$\pi\in \Irr_{\gen}(G)$ on a supercuspidal line~$\Z_\rho$,~$\pi$ can be written 
in a unique manner~$\pi_{\tnb}\times \pi_{\b}$ as in \cite[Proposition 2.3]{KM17}. The representation 
$\pi_{\b}$ is a banal 
representation, which can be written in a unique manner as a (possibly empty) product 
\[\pi_{\b}=\prod_{i=1}^s \La([c_i,d_i]_\rho)=\prod_{i=1}^s \St(d_i-c_i+1,\nu^{c_i}\rho)\] with the segments~$[c_i,d_i]_\rho$ unlinked (in particular all 
lengths~$b_i-a_i+1$ are~$<o(\rho)$, hence the product is empty if~$\rho$ is non banal). The representation~$\pi_{\tnb}$ is a (possibly empty) product of the form 
\[\pi_{\tnb}=\prod_{k=0}^r \La([0,a_k-1]_{\St_k(\Z_\rho)})=\prod_{k=0}^r \St(a_k,\St_k(\Z_\rho))\] for 
the~$0\leq a_k < \ell$. 

By \cite[III.5.10]{Vbook}, if~$\rho\in \Irr_{\c}(G_n,\Fl)$, there is~$\tilde{\rho}\in \Irr_{\c}(G_n,\Ql)_e$ such that 
$r_\ell(\tilde{\rho})=\rho$.

\subsection{The $\V$-correspondence}\label{section Vcorresp}

In \cite[I.8.4]{Viginv}, Vignéras introduces a surjection \[J_\ell:\Irr(G_n,\Ql)_e\rightarrow \Irr(G_n,\Fl).\] 
Take $\psi$ a (necessarily integral) character of $F$ such that $r_\ell(\psi)$ is nontrivial. For $A$ a finite subset of $\{1,\dots,n-1\}$, we denote by $\psi_{A}$ the (degenerate when $A\neq \emptyset$) character of $N_n$ defined by
 \[\psi_{A}(n)=\sum_{i\notin A} n_{i,i+1}.\]
Then by \cite[Theorem 8.2]{Z}, for $\pi \in \Irr(G_n,\Ql)_e$, there is a unique $A$ such that $\pi$ has a Whittaker model (which is unique) with respect to $\psi_A$ (we will say of type $A$). By \cite[Proposition 9.19]{MSDuke}, the reduction 
modulo $\ell$ of $\pi$ has a unique irreducible summand $\pi'$ which has a Whittaker model with respect to $r_\ell(\psi)_A$. The map $J_\ell$ is then defined by $J_\ell(\pi)=\pi')$. Let's compute some examples.

\begin{ex}\label{Jexample}
\begin{itemize}
\item If $\D=[a,b]_{\tau}$ is a cuspidal segment such that $\rho:=r_\ell(\tau)$ is cuspidal, and we set 
$r_\ell(\D)=[a,b]_{\rho}$, then \[J_\ell(\Ze(\D))=\Ze(r_\ell(\D)).\]
\item Take $\tau\in \Irr_{\c}(G,\Q_l)_e$ such that 
$\rho_0:=r_\ell(\tau)$ is a cuspidal non supercuspidal representation of $G$. Then 
$\rho_0=\St_r(\Z_\rho)$. We set 
\[\rho_i:= \St_{r+i}(\Z_\rho).\] 
For 
$k\in \N-\{0\}$, write the $\ell$-adic expansion of $k$:
\[k=a_0+a_1\ell+\dots+a_d\ell^d,\] then:
\[J_\ell(\St(k,\tau))=\St(a_0,\rho_0)\times \St(a_1, \rho_1) \times \dots \times \St(a_d, \rho_d).\] 

\item Take $\tau\in \Irr_{\c}(G,\Ql)_e$ such that 
$\rho:=r_\ell(\tau)$ is supercuspidal. Take $k\geq 1$ and write the euclidean division of $k$ by $o(\rho)$:
\[k=u o(\rho)+r.\]
Again set \[\rho_i:= \St_i(\Z_\rho)\] and write the $\ell$-adic expansion of $u$:
\[u=a_0+a_1\ell+\dots+a_d\ell^d.\] Then:
\[J_\ell(\St(k,\tau))=\St(\rho,r)\times \St(\rho_0,a_0)\times \St(\rho_1, a_1) \times \dots \times \St(\rho_d,a_d).\]
\end{itemize}
\end{ex}

An element~$\Phi$ of~$\Rep_{\ss}(\WD,\Ql)_e=[\Nilp_{\ss}(\WD,\Ql)_e]$ supported on an irreducible line 
can be written in a unique manner under the form
\[\bigoplus_{i\geq 1}\bigoplus_{k\in \Z} a_{i,k} [0,i-1]\otimes \nu^k\Theta\] for~$\Theta\in \Irr(\W_F,\Ql)_e$ with all~$a_{i,k}$ except possibly a finite number being zero. 

By definition, we set \[r_\ell(\Phi)= \bigoplus_{i\geq 1}\bigoplus_{k\in \Z} a_{i,k} [0,i-1]\otimes \nu^k r_\ell(\Theta),\] 
where we recall that~$r_\ell(\Theta)$ is either irreducible, or of the form \[r_\ell(\Theta)=\ell^a({\textstyle\bigoplus_{k=0}^{o(\Psi)-1} }\nu^k \Psi)\] for~$a\geq 0$ and~$\Psi \in \Irr(\W_F,\Fl)$ thanks to Proposition \ref{reduction of irreps simpler}.

We now recall one of the main results of \cite{Viginv}, which is the~$\ell$-modular local Langlands correspondence. 
We denote by~$\LLC$ the~$\ell$-adic Langlands correspondence from~$\Nilp_{\ss}(\WD,\Ql)$ to~$\Irr(G,\Ql)$.

\begin{thm}\cite[Theorem 1.6 and 1.8.5]{Viginv}\label{Vcorrespondence}
There is a bijection \[\V:\Irr(G,\Fl)\simeq \Nilp_{\ss}(\WD,\Fl),\] characterized by the property 
\[\V(\J_\ell(\LLC(\Phi)^*)^*)=r_\ell(\Phi)\] for any \[\Phi\in \Nilp_{\ss}(\WD,\Ql)_e.\] 
It induces a bijection between~$\Irr_{\sc}(G,\Fl)$ and~$\Irr(\W_F,\Fl)$.
\end{thm}

The following immediate properties of~$\V$ are clear, though not explicitly stated in \cite{Viginv}:

\begin{lemma}\label{commutation with basic ops}
The bijection~$\V$ commutes with character twists, takes the central character to the determinant, and 
commutes with taking duals. Moreover if \[\pi=\prod_{i=1}^r\pi(\Z_{\rho_i})\] with 
$\pi(\Z_{\rho_i})$ supported on the supercuspidal line~$\Z_{\rho_i}$ and~$\Z_{\rho_i}\neq \Z_{\rho_j}$ for~$i\neq j$, then 
\[\V(\pi)=\bigoplus_{i=1}^r \V(\pi(\Z_{\rho_i})).\]
\end{lemma}
\begin{proof}
Both modular and~$\ell$-adic Aubert-Zelevinsky involutions,~$J_{\ell}$ and~$\LLC$ commute with character twists, and 
\[r_\ell:\Nilp_{\ss}(\WD,\Ql)_e\rightarrow \Nilp_{\ss}(\WD,\Fl),\] as well, hence the first statement.  
For~$\pi\in \Irr(G,\Ql)_e$, the central character~$c_{J_\ell(\pi)}$ of~$J_\ell(\pi)$ is equal to~$r_\ell(c_\pi)$, both Zelevinsky involutions do 
not touch the central character, and~$\LLC$ takes determinant to central character. The commutation with taking duals also from the 
fact that~$r_\ell$, both Aubert-Zelevinsky involutions,~$\LLC$ and~$J_\ell$ share this property. The last property is a consequence of the similar property for 
$\LLC$, the fact that 
both~$^*$-involutions commute with irreducible parabolic induction, and the fact 
that~$J_\ell(\pi_1\times \pi_2)=J_{\ell}(\pi_1)\times J_\ell(\pi_2)$ when~$\pi_1$ and~$\pi_2$ have disjoint cuspidal supports. 
\end{proof}

We end this section with some examples of~$\V$-parameters.

\begin{ex}\label{Vexample}
\begin{itemize}
\item If~$\rho=\St_k(\rho_0)$ is cuspidal non supercuspidal, with 
$\rho_0\in \Irr_{\sc}(G)$ and set \[\Psi_0=\V(\rho_0)\in \Irr(\W_F,\Fl)\] so that 
$o(\rho_0)=o(\Psi_0)$. Then if~$\tilde{\rho}_0$ is a (necessary cuspidal) 
lift of~$\rho_0$, and~$\tilde{\Psi}_0$ is a (necessary irreducible) lift of 
$\Psi_0$, by Theorem \ref{Vcorrespondence}, we have~$r_\ell(\tilde{\Psi}_0)=\Psi_0$. Now we also have 
\[\rho=\J_\ell(\St(o(\rho_0)\ell,\tilde{\rho_0})),\] hence 
\[\rho=\rho^*=\J_\ell(\St(o(\rho_0)\ell,\tilde{\rho_0}))^*=\J_\ell(\Ze([0,o(\rho_0)\ell^k-1]_{\tilde{\rho}_0})^*)^*.\]
On the other hand
 \[r_\ell(\LLC(\Ze([0,o(\rho_0)\ell^k-1]_{\tilde{\rho}_0}))=
 r_\ell({\textstyle\bigoplus_{i=0}^{o(\Psi_0)\ell^k-1}} \nu^i\tilde{\Psi}_0)=\ell^k({\textstyle\bigoplus_{i=0}^{o(\Psi_0)-1}}\nu^i \Psi_0),\] 
 hence
 \[\V(\rho)=\ell^k({\textstyle\bigoplus_{i=0}^{o(\Psi_0)-1}}\nu^i \Psi_0).\]
\item If~$\rho$ is supercuspidal, and~$\Psi=\V(\rho)\in \Irr(\W_F,\Fl)$, then \[\V(\La([a,b]_{\rho}))=[a,b]\otimes \Psi.\] 
Indeed start with an~$\ell$-adic lift~$\tilde{\rho}$ of~$\rho$ with Langlands parameter~$\tilde{\Psi}$ so
that~$\Psi=r_\ell(\tilde{\Psi})$. Then 
\[J_\ell(\La([a,b]_{\tilde{\rho}})^*)^*= J_\ell(\Ze([a,b]_{\tilde{\rho}}))^*=\Ze([a,b]_{\rho})^*=\La([a,b]_{\rho})\] 
but on the other hand 
\[r_\ell(\La([a,b]_{\tilde{\rho}})=r_\ell([a,b]\otimes \tilde{\Psi})= [a,b]\otimes \Psi.\]
Similarly if~$\rho=St_k(\rho_0)$ is cuspidal non supercuspidal with~$\V(\rho_0)=\Psi_0$, we find 
\[\V(\La([a,b]_{\rho}))=\ell^k [a,b]\otimes ({\textstyle\bigoplus_{i=0}^{o(\Psi_0)-1}}\nu^i \Psi_0).\] 
\item Take~$\pi\in \Irr_{\gen}(G,\Fl)$, supported on the supercuspidal line~$\Z_\rho$, and that~$\V(\rho)=\Psi$. Then according to 
\cite[Proposition 2.3]{KM17}, it can be written~$\pi_{\b}\times \pi_{\tnb}$ where~$\pi_b$ is banal, and 
no segment occurring in~$\pi_{\tnb}$ is banal ($\pi_{\tnb}$ is \emph{totally non-banal}). Write 
\[\pi_{\b}=\prod_{i\geq 1}\prod_{k=0}^{o(\Psi)-1} \La([0,i-1]_{\nu^k\rho})^{c_{i,k}},\]
where the occurring segments are unlinked, hence 
in particular for each fixed~$i$, there is a~$k$ such that~$c_{i,k}=0$.
Write \[\pi_{\tnb}=\prod_{k\geq 0} \La([0,a_k-1]_{\St_k(\rho)})\] with~$0\leq a_k < \ell$. The using a generic 
\textit{standard lift} of~$\pi$ as in \cite[Definition 2.24]{KM17}, one checks that 
\[\V(\pi)=\V(\pi_{\b})\oplus \V(\pi_{\tnb}),\] that 
\[\V(\pi_{\b})=  \bigoplus_{i\geq 1}\bigoplus_{k= 0}^{o(\Psi)-1} c_{i,k} [0,i-1]\otimes \nu^k\Psi \] and that 
\[\V(\pi_{\tnb})=\bigoplus_{j\geq 0} \ell^j [0,a_j-1]\otimes \bigoplus_{k=0}^{o(\Psi)-1} \nu^k \Psi.\] 

\end{itemize}
\end{ex}

\subsection{The $\C$-correspondence}\label{section C corresp}

A $\V$-parameter $\Phi$, by definition in $\Nilp_{\ss}(\W_D,\Fl)$, supported on an irreducible line 
$\Z_{\Psi}$ can be uniquely written in its \textit{standard form} as 
\[\bigoplus_{i\geq 0}\bigoplus_{k=0}^{o(\Psi)-1} a_{i,k} [0,i-1]\otimes \nu^k\Psi.\] 
We say that it is \textit{acyclic} if for each fixed $i$, there is $0\leq k \leq o(\Psi)-1$ such that $a_{i,k}=0$. We say that it is \textit{cyclic} if for each fixed $i$, the coefficient $a_{i,k}$ is independent of $k$.

Take a general $\V$-parameter as above, and set: \[b_i=\mathrm{min}_{k} a_{i,k},\] and \[c_{i,k}=a_{i,k}-b_i.\] It can then be rewritten as 
\[\Phi=\Phi_{\acyc}\oplus \Phi_{\cyc},\]
with \[\Phi_{\acyc}=\bigoplus_{i\geq 1}\bigoplus_{k= 0}^{o(\Psi)-1} c_{i,k} [0,i-1]\otimes \nu^k\Psi\] and  
\[\Phi_{\cyc}=\bigoplus_{j\geq 1} b_j [0,j-1]\otimes \bigoplus_{k=0}^{o(\Psi)-1}\nu^k\Psi.\] 
Notice that for each $i$, we have $c_{i,k}=0$ for one $k$ so it makes sense to call $\Phi_{\acyc}$ the \emph{acyclic part} of $\Phi$, and 
we call $\Phi_{\cyc}$ its \emph{cyclic part}.
Conversely, if a $\V$-parameter $\Phi$ is written as the sum of an acyclic and a cyclic parameter: 
\[\left(\bigoplus_{i\geq 1}\bigoplus_{k= 0}^{o(\Psi)-1} c_{i,k} [0,i-1]\otimes \nu^k\Psi\right) 
\oplus \left(\bigoplus_{j\geq 1} b_j [0,j-1]\otimes \bigoplus_{k=0}^{o(\Psi)-1}\nu^k\Psi\right),\] then its standard form is equal to 
\[\bigoplus_{i\geq 1}\bigoplus_{k=0}^{o(\Psi_0-1)} a_{i,k} [0,i-1]\otimes \nu^k\Psi\] with  
$a_{i,k}=c_{i,k}+b_i$ thus the decomposition of $\Phi$ as the direct sum of a cyclic and acyclic parameter is unique, 
and it is 
$\Phi=\Phi_{\acyc}\oplus \Phi_{\cyc}$.

We now define an injection $\CV$ of $\Nilp_{\ss}(\W_D,\Fl)$ into $[\Rep_{\ss}(\W_D,\Fl)]$, which is not the natural inclusion.

\begin{definition}\label{CVdefinition}
Take $\Phi=\Phi_{\acyc}\oplus \Phi_{\cyc} \in \Nilp_{\ss}(\W_D,\Fl)$ supported on an irreducible line $\Z_\Psi$, 
and write \[\Phi_{\cyc}=\bigoplus_{j\geq 1} b_j [0,j-1]\otimes \bigoplus_{k=0}^{o(\Psi)-1}\nu^k\Psi.\] 
 We set \[\CV(\Phi_{\cyc})=\bigoplus_{j\geq 1} b_j [0,j-1]\otimes C(\Z_\Psi)\in [\Rep_{\ss}(\W_D,\Fl)],\] and then 
 \[\CV(\Phi)=\Phi_{\acyc}\oplus \CV(\Phi_{\cyc}) \in [\Rep_{\ss}(\W_D,\Fl)].\]
 Finally, if $\Phi=\bigoplus_j \Phi(\Z_{\Psi_j})\in \Nilp_{\ss}(\W_D,\Fl)$ where $\Phi(\Z_{\Psi_j})$ is supported on the irreducible line $\Z_{\Psi_j}$, 
 and $\Z_{\Psi_k}\neq \Z_{\Psi_l}$ for $k\neq l$, we set 
 \[\CV(\Phi)=\bigoplus_j \CV(\Phi(\Z_{\Psi_j})).\]
\end{definition}

We have the following immediate lemma.

\begin{lemma}
The map $\CV:\Nilp_{\ss}(\W_D,\Fl)\rightarrow [\Rep_{\ss}(\W_D,\Fl)]$ is injective.
\end{lemma}

We can define thanks to $\V$ and $\CV$, and injection $\C$ of $\Irr(G,\Fl)$ into $[\Rep_{\ss}(\W_D,\Fl)]$.

\begin{definition}
For $\pi\in \Irr(G)$, we set $\C(\pi)=\CV(\V(\pi))$.
\end{definition}

We do the $\C$-version of example \ref{Vexample}.

\begin{ex}\label{Cexample}
\begin{itemize}
\item If $\pi\in \Irr_{\sc}(G)$ and $\V(\pi)=\Psi)$, then \[\C(\pi)=\Psi\] if $\pi$ is banal, 
and \[\C(\pi)=C(\Z_\Psi)=C(\{\Psi\})\] if $\pi$ is non banal.

\item If $\rho=\St_k(\rho_0)$ is cuspidal non supercuspidal with 
\[\rho_0\in \Irr_{\sc}(G)\] and $\Psi_0=\V(\rho_0)\in \Irr(\W_F,\Fl)$, then 
\[\C(\rho)=\ell^k C(\Z_{\Psi_0}).\]

\item If $\rho$ is supercuspidal, and $\Psi=\V(\rho)\in \Irr(\W_F,\Fl)$, then 
\[\C(\La([a,b]_{\rho}))=[a,b]\otimes \Psi\] if $\pi$ is banal and 
\[\C(\La([a,b]_{\rho}))=[a,b]\otimes C(\Z_\Psi)\] if $\pi$ is non banal.

\item Take $\pi=\pi_{\b}\times \pi_{\tnb}\in \Irr_{\gen}(G,\Fl)$, supported on the suppercuspidal line $\Z_\rho$ with
$\V(\rho)=\Psi$. Write 
\[\pi_{\b}=\prod_{i\geq 1}\prod_{k=0}^{o(\Psi)-1} \La([0,i-1]_{\nu^k\rho})^{c_{i,k}},\] 
where the occuring segments are unlinked, hence 
in particular for each fixed $i$, there is a $k$ such that $c_{i,k}=0$.
Write \[\pi_{\tnb}=\prod_{k\geq 0} \La([0,a_k-1]_{\St_k(\rho)})\] with $0\leq a_k < \ell$. Then 
\[\V(\pi)_{\acyc}=\V(\pi_{b})\] and \[\V(\pi)_{\cyc}=\V(\pi_{\tnb}).\]
Hence \[ \C(\pi)=\C(\pi_{\b})\oplus \C(\pi_{\tnb})\] where 
\[\C(\pi_{\b})=\bigoplus_{i\geq 1}\bigoplus_{k= 0}^{o(\Psi)-1} c_{i,k} [0,i-1]\otimes \nu^k\Psi \] and  
\[\C(\pi_{\tnb})=\bigoplus_{j\geq 0} \ell^j [0,a_j-1]\otimes C(\Z_{\Psi}).\] 

\end{itemize}
\end{ex}

\subsection{Preservation of local constants}

It is an immediate verification to check that $\CV$ commutes with taking duals, direct sums, twisting by characters, and does not change the determinant. Hence the correspondence $\C$ shares with $\V$ the properties of Lemma \ref{commutation with basic ops}. 
Hence for the moment we lost nothing introducing $\C$, but we gained nothing neither. However there is 
one important property that the $\V$ correspondence does not share with the $\LLC$, which is 
the preservation of local constants. For the above sentence to make sense, one must have a definition of local factors for elements in $\Irr(G,\Fl)$. Indeed there is one: for standard local factors, they have been defined in \cite{MinguezZeta} (the so called Godement-Jacquet method), and for $L$-factors of pairs, they have been defined in \cite{KM17} for pairs of generic representations (the Rankin-Selberg convolution method of Jacquet, Piatetski-Shapiro and Shalika). It should be true that 
$L^{\RS}(X,\pi,\1)=L^{\GJ}(X,\pi)$, and similarly for $\gamma$ and $\e$ factors for $\pi\in \Irr_{\gen}(G)$, but we did not check this. In what follows we will only consider Rankin-Selberg $L$-factors defined in \cite{KM17}, and we will 
drop the $\mathrm{RS}$ exponent. 
 
We claim that $\C$ preserves local factors of pairs, whereas $\V$ does not. Let us give a basic example where we consider the $L$-factor only.

\begin{ex}\label{non preservation of local factors}
Consider the cuspidal representation $\rho=\St_0(\1)$ of $\G_n$. Then one has \[\V(\rho)=\bigoplus_{k=0}^{o(\nu)-1} \nu^k\] whereas 
 \[\C(\rho)=C(\Z_{\1}).\] According to \cite[Theorem 4.9]{KM17}, One has 
 \[L(X,\pi):=L(X,\pi,\1)=1.\] By definition of the Deligne $L$-factor, one also has \[L(X,C(\Z_{\1}))=1\] because 
 the Deligne operator associated to $C(\Z_{\1})$ is bijective. However, 
 \[L(X,{\textstyle\bigoplus_{k=0}^{o(\nu)-1} \nu^k})=\prod_{k=0}^{o(\nu)-1}\frac{1}{1-q^{-k} X}.\]
\end{ex}

We can finally prove the central result of this paper. We fix a nontrivial character $\psi$ of $F$ with values in 
$\Fl^\times$, and $\tilde{\psi}$ a lift of $\psi$.

\begin{thm}\label{preservation of local factors}
For $\pi,\ \pi'\in \Irr_{\gen}(G,\Fl)$, then:
\[\gamma(X,\C(\pi)\otimes_{\ss} \C(\pi'),\psi)=\gamma(X,\pi,\pi',\psi),\]
\[L(X,\C(\pi)\otimes_{\ss} \C(\pi'))=L(X,\pi,\pi'),\] hence 
\[\e(X,\C(\pi)\otimes_{\ss} \C(\pi'),\psi)=\e(X,\pi,\pi',\psi).\]
\end{thm}
\begin{proof}
Let's prove the statement on~$\gamma$ factors first. Let $\pi$ and $\pi'$ belong to $\Irr_{\gen}(G)$ and set 
$\tilde{\pi}$ and $\tilde{\pi}'$ two~$\ell$-adic generic representations such that $\pi=\J_\ell(\tilde{\pi})$ and 
$\pi'=\J_\ell(\tilde{\pi'})$ (for example standard lifts as in \cite[Definition 2.24]{KM17}). Then according to 
\cite[Theorem 3.13]{KM17}, one has \[\gamma(X,\pi,\pi',\psi)=r_\ell(\gamma(X,\tilde{\pi},\tilde{\pi}',\tilde{\psi})).\]
Set $\tilde{\Phi}$ and 
$\tilde{\Phi}'$ be the semisimple representations of $\W_F$ corresponding to the supercuspidal support of 
$\tilde{\pi}$ and $\tilde{\pi}'$ via $\LLC$ (i.e. $\LLC(\tilde{\pi})=(\tilde{\Phi},\star)$ and $\LLC(\tilde{\pi}')=(\tilde{\Phi}',\star)$). The $\LLC$ and the standard properties of $\gamma$-factors tell us that 
\[\gamma(X,\tilde{\pi},\tilde{\pi}',\tilde{\psi})= \gamma(X,\tilde{\Phi}\otimes \tilde{\Phi}',\tilde{\psi}).\]
Set $r_\ell(\tilde{\Phi})=\Phi$ and $r_\ell(\tilde{\Phi'})=\Phi'$, hence $r_\ell(\tilde{\Phi}\otimes \tilde{\Phi'})=\Phi\otimes_{\ss} \Phi'$. According to Theorem \ref{Vcorrespondence}, or more simply \cite[Theorem 1.6]{Viginv}
which states that the semisimple $\LLC$ commutes with reduction modulo $\ell$, we deduce that $\Phi$ corresponds to the supercuspidal 
support of $\pi$, whereas $\Phi'$ corresponds to the supercuspidal support of $\pi'$, i.e. 
$\C(\pi)=[\Phi,\star]$ and $\C(\pi')=[\Phi',\star]$, so that 
\[\gamma(X,\C(\pi)\otimes_{\ss} \C(\pi'),\psi)=\gamma(X,\Phi\otimes_{\ss} \Phi',\psi).\] However 
\[\gamma(X,\Phi\otimes_{\ss} \Phi',\psi)=r_\ell(\gamma(X,\tilde{\Phi}\otimes \tilde{\Phi'},\tilde{\psi}))\] according to 
Proposition \ref{HM}. This ends the proof of the assertion on $\gamma$-factors. It remains to prove that on $L$-factors, the statement on $\e$ will follow.

Thanks to the discussion before \cite[Proposition 2.3]{KM17}, we write $\pi=\pi_{\b}\times \pi_{\tnb}$ and $\pi'=\pi'_{\b}\times \pi'_{\tnb}$, and \cite[Theorem 4.19]{KM17} 
tells us that \[L(X,\pi,\pi')=L(X,\pi_{\b},\pi'_{\b}).\] Set  
$[\Phi_{\b},U_{\b}]=\C(\pi_{\b})$, $[\Phi_{\tnb},U_{\tnb}]=\C(\pi_{\tnb})$,
 $[\Phi'_{\b},U'_{\b}]=\C(\pi'_{\b})$, $[\Phi'_{\tnb},U'_{\tnb}]=\C(\pi'_{\tnb})$, in particular 
 $U_{\tnb}$ and $U'_{\tnb}$ are bijective whereas $U_{\b}$ and $U'_{\b}$ are nilpotent. One has 
 \[\C(\pi)\otimes_{\ss} \C(\pi')=\] 
 \[\C(\pi_{\b})\otimes_{\ss} \C(\pi'_{\b})\oplus \C(\pi_{\b})\otimes_{\ss} \C(\pi'_{\tnb}) 
 \oplus \C(\pi_{\tnb})\otimes_{\ss} \C(\pi'_{\b})\oplus \C(\pi_{\tnb})\otimes_{\ss} \C(\pi'_{\tnb}).\]
 If one writes any of the latter three direct sums under the form $[\Phi, U]$, then $U$ is bijective (see in particular remark \ref{tensor product preserves non banal}). This implies that 
\[L (X,\C(\pi)\otimes_{\ss} \C(\pi'))=L (X,\C(\pi_{\b})\otimes_{\ss} \C(\pi'_{\b})).\]
Hence it remains to prove the equality:
\[L (X,\C(\pi_{\b})\otimes_{\ss} \C(\pi'_{\b}))=L(X,\pi_{\b},\pi'_{\b}) \]
Now appealing to \cite[Theorem 4.12 and Corollary 4.20]{KM17}, the multplicativity of Deligne $L$-factors with respect to direct sums and the multiplicativity relation of Lemma \ref{multiplicativity} show that it is enough to prove it for 
$\pi_b=\rho$ and $\pi'_{\b}=\rho'$ banal supercuspidal representations. Take $\tilde{\rho}$ (resp. $\tilde{\rho}'$) a cuspidal lift of $\rho$ (resp. $\rho'$), so that $\tilde{\Psi}:=\LLC(\tilde{\rho})$ (resp. 
$\tilde{\Psi'}:=\LLC(\tilde{\rho}')$) is an irreducible lift of $\Psi:=\C(\rho)$ (resp. $\Psi':=\LLC(\rho')$). 
Then \[L (X,\rho,\rho')=r_\ell(L(X,\tilde{\rho},\tilde{\rho}'))\] by \cite[Theorem 4.18]{KM17}, 
\[L (X,\Psi\otimes_{\ss}\Psi')=r_\ell(L(X,\tilde{\Psi}\otimes \tilde{\Psi'}))\] thanks to Theorem \ref{banal Deligne L functions}, and \[L(X,\tilde{\rho},\tilde{\rho})=L(X,\tilde{\Psi}\otimes \tilde{\Psi}')\] by the $\ell$-adic $\LLC$. This shows the following equality and ends the proof: \[L (X,\rho,\rho')=L (X,\Psi \otimes_{\ss} \Psi' ).\]
\end{proof}

\begin{rem}
As we said we leave for later the equality $L^{\GJ}=L^{\RS}$ on generic representations. We also believe 
$\C$ sends the Godement-Jacquet local factors of Minguez on $\Irr(G,\Fl)$ to the standard local factors on $[\Rep_{\ss}(\WD,\Fl)]$. This can be easily checked for $G_2$ by the calculations carried out in \cite{MinguezZeta}, we leave the general case for a further investigation.
\end{rem}

\bibliographystyle{plain}
\bibliography{Modlfactors}

\end{document}